\documentclass[reqno]{amsart}
\usepackage[T1]{fontenc}
\usepackage[latin1]{inputenc}
\usepackage[english]{babel}
\usepackage[babel]{csquotes}

\usepackage{cite}
\usepackage{amssymb}
\usepackage{amsmath}
\usepackage{amsthm}
\usepackage{latexsym}
\usepackage{graphicx}
\usepackage{mathrsfs}
\usepackage{mathrsfs}
\usepackage{bbm}
\usepackage{verbatim}
\usepackage[colorlinks=true, pdfstartview=FitV, linkcolor=blue, 
citecolor=blue, urlcolor=blue]{hyperref}

\usepackage[usenames,dvipsnames]{color}

\newtheorem{thm}{Theorem}[section]
\newtheorem{cor}[thm]{Corollary}
\newtheorem{lem}[thm]{Lemma}
\newtheorem{prop}[thm]{Proposition}

\def\enne{\mathbb{N}}
\def\zeta{\mathbb{Z}}

\def\erre{\mathbb{R}}
\def\Rz{\mathbb{R}}

\def\P{\mathbb{P}}

\def\E{\mathop{{}\mathbb{E}}}
\def\I{\mathcal I}

\def\cL{\mathscr{L}}
\def\cF{\mathscr{F}}

\def\eps{\varepsilon}
\def\cP{\mathscr{P}}

\def\OO{\mathcal{O}}
\renewcommand{\d}{{\mathrm d}}

\def\beq{\begin{equation}}
\def\eeq{\end{equation}}

\def\to{\rightarrow}
\def\wto{\rightharpoonup}
\def\wstarto{\stackrel{*}{\rightharpoonup}}
\def\embed{\hookrightarrow}

\def\norm #1{\left\|#1\right\|}

\def\sp #1#2{\left<#1,#2\right>}
\newcommand\ip\sp

\newcommand{\haz}{\widehat}

\begin{document}
\title[Doubly nonlinear stochastic evolution equations II]
{Doubly nonlinear stochastic evolution equations II}

\author{Luca Scarpa}
\address[Luca Scarpa]{Faculty of Mathematics, University of Vienna, 
Oskar-Morgenstern-Platz 1, 1090 Vienna, Austria.}
\email{luca.scarpa@univie.ac.at}
\urladdr{http://www.mat.univie.ac.at/$\sim$scarpa}

\author{Ulisse Stefanelli}
\address[Ulisse Stefanelli]{Faculty of Mathematics,  University of Vienna, Oskar-Morgenstern-Platz 1, A-1090 Vienna, Austria,
Vienna Research Platform on Accelerating Photoreaction Discovery, University of Vienna, W\"ahringer Str. 17, 1090 Vienna, Austria,$\&$ Istituto di Matematica Applicata e Tecnologie Informatiche ``E. Magenes'' - CNR, v. Ferrata 1, I-27100 Pavia, Italy}
\email{ulisse.stefanelli@univie.ac.at}
\urladdr{http://www.mat.univie.ac.at/$\sim$stefanelli}

\subjclass[2010]{35K55, 35R60, 60H15}

\keywords{Doubly nonlinear stochastic equations, 
strong and martingale solutions, existence, 
maximal monotone operators, generalized It\^o's formula}   

\begin{abstract}
Complementing the analysis in \cite{ScarStef-SDNL}, 
we investigate the well-posedness of SPDEs problems of doubly
nonlinear type. These arise ubiquitously in
the modelization of dissipative media and correspond to generalized balance laws
between conservative and nonconservative dynamics. We extend the reach
of the classical deterministic case by allowing for stochasticity. The
existence of martingale solutions is proved via a regularization
technique, hinging on the validity of a It\^o formula in a minimal
regularity setting. Under additional assumptions, the well-posedness of
stochastically strong solutions is also shown. 
\end{abstract}

\maketitle


\section{Introduction}
\setcounter{equation}{0}
\label{sec:intro}

This is the second paper of a series devoted to the existence of solutions to
stochastic doubly nonlinear evolution equations. Complementing the
analysis of \cite{ScarStef-SDNL}, we focus here on equations which, 
in the {\it deterministic} case, take the form
\begin{equation}
A(\partial_t u(t) ) + \partial \haz B(u(t)) \ni
F(t,u(t))\,.\label{eq:biot}
\end{equation}
Here, $t\in [0,T] \mapsto u(t)\in H$ represents the state of a
physical systems at time $t$, seen as an element in the separable Hilbert space
$H$, $T>0$ is some final reference time, and $\partial_t$ is the time
derivative. The functional $\haz B: H \to (-\infty,+\infty]$ denotes
the {\it energy} of the system, here assumed to be convex, and
$\partial \haz B$ is the corresponding subdifferential. The term $t
\mapsto F(t,u(t)) \in H$ represents further external actions. The possibly
multivalued, maximal monotone operator $A:H \to 2^H$ with $0\in A(0)$ corresponds to
{\it dissipation} instead. More precisely, the scalar $(A(\partial_t
u), \partial_t u)_H\geq 0$ represents the instantaneous
dissipation of the system corresponding to the rate $\partial_t u$, 
where $(\cdot,\cdot)_H$ is the scalar product in $H$.

Given these provisions,
relation \eqref{eq:biot} is nothing but the balance between 
{\it conservative} actions, modelled by $\partial \haz B(u)$, 
{\it dissipative} actions, encoded in $A(\partial_tu)$, 
and {\it external} actions, namely $F(t,u)$.  By
additionally letting $F(t,u) = \partial_u\haz F(t,u)$, where
$\haz F(t,u)$ denotes the work of external actions at time $t$
on the state $u\in H$, and assuming
enough smoothness, one can check in particular the
{\it dissipation inequality}
\begin{equation}
\frac{\d}{\d t}\big(\haz B(u) - \haz F(\cdot,u)\big)  + \partial_t \haz F(\cdot,u)= (\partial \haz B(u) -
F(\cdot,u),\partial_t u )_H   = -(A(\partial_t
u), \partial_t u)_H \leq 0\,.\label{eq:dissipativity}
\end{equation}
The left-hand side of this relation features the sum of the variation
of the free
energy of the system $\haz B(u) - \haz F(t,u)$, 
namely the stored part of the energy $\haz B$
which is not contributed from the exterior,
plus the power of external actions. In particular,
\eqref{eq:dissipativity} entails that this balance is negative along
all trajectories, expressing indeed the dissipativity of the
system. Indeed, the generalized balance \eqref{eq:biot} is prototypical of
dissipative systems and 
arises
ubiquitously in applications, especially in relation with
Thermomechanics and dissipative media: 
see  \cite{Moreau70,Moreau71}
and \cite{Germain73}
for justifications. Note that 
the existence of a potential $\haz F(t,u)$ is instrumental to
obtaining relation \eqref{eq:dissipativity}, 
being however not required for the existence theory.

The linear case where $A$ is chosen to be the identity in $H$ (not
specifically in focus here,
however), corresponds to perturbed gradient flows and is rather classical \cite{Ambrosio}.  
For existence
results for $A$ nonlinear and coercive the reader is referred to {\sc Colli \& Visintin}
\cite{Colli-Vis90} and {\sc Colli} \cite{Colli92ter}  (see also  \cite{Barbu75},  
\cite{Arai79}, and  \cite{Senba86} among others). Besides
existence, structural stability \cite{AiziYan}, perturbations and
long-time behavior \cite{G11,G13,Roubicek,SSS,Segatti}, and
variational characterization of solutions
\cite{AS1,AS2,Ste08a} have also been considered.  The
case $A$ nonlinear and $0$-homogeneous (not included in our analysis)
is related with the modeling of rate-independent evolution: see
\cite{Mielke-Roubicek} for a recent comprehensive collection of
results.

Our aim is to study a {\it stochastic} version of equation
\eqref{eq:biot}, namely \eqref{eq'} below, 
taking into account possible random disturbances occurring 
in the evolution and affecting the energy--dissipation balance.
Let a
filtered probability space $(\Omega,\cF,(\cF_t)_{t\in[0,T]},\P)$  and
a cylindrical Wiener process $W$ on a separable Hilbert space
$U$ be given.
Solutions $(\omega,t) \mapsto u(\omega,t)\in H$ are assumed to be It\^o
processes of the form 
\beq\label{ito_process}
  u(t)=u^d(t) + \int_0^tu^s(s)\,\d W(s)\,, \quad t\in[0,T]\,,
\eeq
where $u^d$ is an absolutely continuous process
and $u^s$ is a $W$-stochastically integrable process. 
For processes in the form \eqref{ito_process}, 
the averaged variation rate is given exactly by 
the contribution $\E\partial_t u^d$, where
$\E$ denotes expectation.
It is then natural to consider 
the differentiable process $u^d$ as the one accounting for dissipation, 
whereas the stochastic integral of $u^s$ has to be interpreted 
as the random perturbation of
the trajectory of $u$. 
Consequently, as we are interested in nonlinear dissipation evolutions, 
we prescribe the action of the dissipation $A$ directly on 
the rate $\partial_t u^d$, by keeping the contribution $u^s\,\d W$ as 
a separate random disturbance affecting the trajectory of $u$.
Bearing in mind these considerations,
we are interested in 
the following doubly nonlinear stochastic
evolution equation
\begin{equation}
  \label{eq'}
  A(\partial_tu^d)\, \d t + u^s \, \d W  + \partial \haz B(u)\, \d t
  \ni F(\cdot,u)\, \d t +
  G(\cdot,u)\, \d W\,,
\end{equation} 
where $G$ is a suitable stochastically integrable operator, possibly depending on $u$.
This may be equivalently seen as a system 
\beq\label{eq'_bis}
  \begin{cases}
  A(\partial_tu^d)  + \partial \haz B(u)
  \ni F(\cdot,u)\,,\\
  u^s=G(\cdot,u)\,,
  \end{cases}
\eeq
under the constraint that $u$ is an It\^o process in the form \eqref{ito_process}.
By taking $G=0$, the latter entails $u^s=0$, so that $u=u^d$, and
the equation reduces to the
deterministic case of \eqref{eq:biot}. 
In the special case where $A$ is the identity in $H$,
the doubly nonlinear problem \eqref{eq'} reduces instead to the classical 
well-understood stochastic evolution equation 
\[
  \d u + \partial \haz B(u)\,\d t \ni F(\cdot,u)\,\d t + G(\cdot,u)\,\d W\,.
\]
More generally, equation \eqref{eq'}, or equivalently \eqref{eq'_bis}, 
features a possibly nonquadratic
dissipation in terms of $\partial_t u^d$, and is the natural 
stochastic version of the deterministic evolution \eqref{eq:biot}.

The consistency of the doubly nonlinear stochastic evolution equation \eqref{eq'}
is particularly evident at the level of the energy--dissipation balance.
Indeed, in Proposition \ref{prop:ito} below we prove a  
refined It\^o formula which would allow us to argue along the lines of   
\eqref{eq:dissipativity} and obtain 
the averaged dissipativity inequality 
\begin{align}
    \frac{\d}{\d t}\E\left(\haz B(u) - \haz F(\cdot,u) - \frac12 \E ({\rm Tr} \, L(\cdot,u))\right)  
    + \E\partial_t \haz F(\cdot,u)=  
    -\E (A(\partial_t u), \partial_t u)_H \leq 0\,.\label{eq:dissipativity2}
\end{align}
Here, $L(\cdot,u) = G (\cdot,u)
G^*(\cdot,u) D_{\mathcal G} \partial \haz B(\cdot)$, where the symbol
$D_{\mathcal G}$ represents 
the G\^ateaux differential. With respect to the deterministic
dissipation inequality \eqref{eq:dissipativity},
its stochastic version \eqref{eq:dissipativity2} features the additional contribution of the
noise, which effectively enters the energy balance via the term in
$L$. 
The dissipativity inequality \eqref{eq:dissipativity2} will be crucial, as it 
allows to obtain uniform estimates on the solutions of the doubly nonlinear 
stochastic problem \eqref{eq'}.

The main result of this note is to prove that the Cauchy problem for equation
\eqref{eq'} admits an analytically strong martingale solution if $A$ is coercive and linearly bounded and the sublevels of $\haz B$ are compact, see
Theorem~\ref{thm1}. The key step is to prove a It\^o formula under
minimal regularity assumptions, see Proposition~\ref{prop:ito}. This in
turn allows us to perform a regularization procedure and identify the
corresponding nonlinear limits by a lower semicontinuity argument.
This technique is new and can be expected be useful elsewhere.
Eventually, if $A$ is strongly monotone and $\partial \haz B$ is linear
and self-adjoint, solutions are actually strong in probability and
unique, see Theorem~\ref{thm2}.

  The main mathematical difficulty in tackling the doubly nonlinear problem 
  \eqref{eq'} is duplex, and consists on the one hand
  in  proving 
  some uniform estimates on the solutions and on the other in identifying 
  the two nonlinearities $A(\partial_tu^d)$ and $\partial \widehat B(u)$.
  These two issues are indeed implicitly related to one another, and 
  their proof hinges on the availability of an accurately refined It\^o's formula in Hilbert spaces.
  First of all, uniform estimates on suitably approximated solutions of \eqref{eq'}
  can be obtained,  provided to show the dissipativity inequality for a 
  regularized equation, which is easier to handle.
  Unlike the deterministic setting, the dissipativity inequality cannot be achieved 
  by simply testing the equation by $\partial_t u^d$, due to the presence of the noise terms.
  Instead, it is obtained by proving an It\^o formula for a suitable 
  regularization of the energy $\widehat B$ and by comparison in the equation itself.
  As this first step is carried out on (relatively) arbitrary regularized problems, 
  one has enough regularity on the approximated energy at disposal 
  and uniform estimates on the solutions can be obtained.
  The delicate point, however, consists in passing to the limit.
  Indeed, as the problem exhibits two nonlinearities, the identification 
  procedure is absolutely nontrivial. The main idea is to identify $\partial B(u)$
  by means of (stochastic) compactness arguments. Once this is done, 
  we characterize $A(\partial_tu^d)$ using lower semicontinuity techniques 
  (the so-called $\limsup$ argument): this however strongly relies on the 
  availability of an energy--dissipation {\it equality} at the limit.
  From the mathematical point of view, we are then forced to 
  prove first an {\it exact} It\^o formula for the energy $\widehat B$
  (not regularized).
  While It\^o {\it inequalities} are available in literature in connection with 
  existence of strong solutions to SPDEs for example \cite{Gess},
  to the best of our knowledge It\^o {\it equalities} associated to 
  a general energy $\widehat B$ have not received attention elsewhere,
  the main problem being the fact that $\widehat B$ is not everywhere defined in $H$,
  having compact sublevels.
  For these reasons, the analytical preliminaries contained here have 
  independent relevance on their own.

To our knowledge, this paper delivers the first existence theory   in    
the doubly nonlinear frame of \eqref{eq'}. In the case of a linear and
nondegenerate operator $A$, existence results can be traced back 
to the classical theory by {\sc Pardoux}\cite{Pard0,Pard} and {\sc Krylov
\&  Rozovski\u \i},\cite{kr}: see  the  monographs \cite{dapratozab,pr}
for a general overview. In particular, the Cauchy problem for
\eqref{eq'} with $A$
linear and $\partial \haz B$ subhomogeneous has been proved to admit 
strong solutions by {\sc Gess}\cite{Gess}.
Well-posedness 
under more relaxed growth conditions are also available, see 
\cite{mar-scar-diss,mar-scar-ref,mar-scar-erg} for semilinear
equations, \cite{mar-scar-note,mar-scar-div, scar-div} for equations in divergence form,
and \cite{barbu-daprato-rock,orr-scar,scar-SCH} for 
porous-media, Allen-Cahn, and Cahn-Hilliard equations.

In the first paper \cite{ScarStef-SDNL} of this series on existence for doubly nonlinear
stochastic evolution equations, we focused on
relations 
of the form
$$\d A(u) + \partial \haz B(u)\, \d t \ni F(\cdot,u)\, \d t + G(\cdot,u)\, \d
W$$
instead. Here, the operator $A$ is applied directly to $u$ and then
differentiated. Such equations are also very relevant from the
applicative viewpoint and arise in connection with different
nonlinear diffusion situations, including phase transitions,  porous-media, the
Hele-Shaw cell, non-Newtonian fluids, and the Mean-Curvature flow.
Existence for the
latter has been obtained in  \cite{SWZ18} (for $A$ bi-Lipschitz) and in
\cite{ScarStef-SDNL} (for $A$ multivalued). 
The case of $A$ nonlinear and $\partial \haz B$ linear has been originally discussed
in \cite{Barbu-DaPrato} in the setting of the
two-phase stochastic Stefan problem, see also 
\cite{BBT14,KM16}.

Let us also point out that in the case where $A=\partial \widehat A$ is a subdifferential 
of a certain convex lower semicontinuous function $\widehat A$ on $H$,
the doubly nonlinear equation \eqref{eq'} fits well
in the framework of the stochastic Weighted Energy--Dissipation approach,
and can be thus seen as limit of convex minimization problems. 
Such technique has been show to be effective when $A$ is the identity 
\cite{ScarStef-WED}. The study of the genuinely doubly nonlinear
case can also be considered.

 We give the detail of our setting and state the existence results in
 Section \ref{sec:main}. Section \ref{sec:ito} is devoted to the proof
 of the above-mentioned generalized It\^o formula. The proof of
 Theorem \ref{thm1}, namely the existence of martingale solutions, is
 detailed in Section \ref{sec:proof1}. Finally, probabilistically strong
 solutions are discussed in Section \ref{sec:proof2}. We conclude in
 Section \ref{sec:appl} by presenting
 an example of a concrete PDE to which the abstract analysis can be applied.


\section{Setting and main results}
\label{sec:main}
In this section we introduce the setting and assumptions
on the data of the problem (Subsections
\ref{ssec:setting}--\ref{ssec:assumptions}) and we present the main
results of the work (Subsection \ref{ssec:main}).

\subsection{Setting and notation}
\label{ssec:setting} 
Let $(\Omega,\cF,(\cF_t)_{t\in[0,T]},\P)$ be a filtered probability space
satisfying the usual conditions, where $T>0$ is a fixed final
reference time, and
let also $W$ be a cylindrical Wiener process on a separable Hilbert space $U$.
We fix once and for all a complete orthonormal system $(u_j)_j$ of~$U$. 

For every Banach spaces $E_1$ and $E_2$, the symbol $\cL(E_1,E_2)$
denotes the space of linear continuous operator from $E_1$ to $E_2$
endowed with the norm-topology.
The space $\cL(E_1,E_2)$ endowed with the strong and weak operator topology 
is denoted by $\cL_s(E_1,E_2)$ and $\cL_w(E_1,E_2)$, respectively.
If $E_1$ and $E_2$ are also Hilbert spaces, the space of trace-class and 
Hilbert-Schmidt operators from $E_1$ to $E_2$ are denoted by 
$\cL^1(E_1,E_2)$ and $\cL^2(E_1,E_2)$, respectively.

In the paper, $V$ is separable reflexive Banach space and $H$ is a
Hilbert space such that $V\embed H$ continuously, densely, and compactly.
As usual, we identify $H$ with its dual, so that we have the Gelfand triple
\[
  V\embed H\embed V^*\,.
\]
Without loss of generality, we suppose that there exists 
a family of regularizing 
linear operators $(R_\lambda)_{\lambda>0}\subset\cL(H,V)$ such that 
$R_\lambda\to I$ in $\cL_s(V,V)$ as $\lambda\searrow0$:
this is trivially satisfied in the majority of interesting examples, for example 
when $V$ is a Sobolev space.

The progressive sigma algebra on $\Omega\times[0,T]$ is denoted by $\cP$.
We use the classical symbols $L^s(\Omega; E)$ and $L^r(0,T; E)$
for the spaces of strongly measurable Bochner-integrable functions 
on $\Omega$ and $(0,T)$, respectively, for all
$s,r\in[1,+\infty]$ and for every Banach space $E$.
If $s,r\in[1,+\infty)$ we use 
$L^s_\cP(\Omega;L^r(0,T; E))$ to indicate
that measurability is intended with respect to $\cP$.
In the case that $s\in(1,+\infty)$, $r=+\infty$, and $E$ is a separable,
we set 
\begin{align*}
  &L^s_\cP(\Omega; L^\infty(0,T; E^*)):=
  \big\{v:\Omega\to L^\infty(0,T; E^*) \text{ weakly*
      progressively measurable}\nonumber\\
&\hspace{48mm} \text{with} \ 
  \E\norm{v}_{L^\infty(0,T; E^*)}^s<\infty
  \big\}\,,
\end{align*}
and recall that by 
\cite[Thm.~8.20.3]{edwards} we have the identification
\[
L^s_\cP(\Omega; L^\infty(0,T; E^*))=
\left(L^{\frac{s}{s-1}}_\cP(\Omega; L^1(0,T; E))\right)^*\,.
\]

\subsection{Spaces of It\^o processes} We introduce here some specific
notation 
for classes of It\^o-type processes. Let $E,F$ be
two separable reflexive Banach spaces and $K$ be a separable Hilbert space, 
such that $E,K\embed F$ continuously. For every $r_1,r_2\in[1,+\infty)$
we use the notation 
\[
  \I^{r_1,r_2}(E,K):=L^{r_1}_\cP(\Omega; W^{1,r_1}(0,T; E)) \oplus 
  L^{r_2}_\cP(\Omega; L^2(0,T; \cL^2(U,K)))\cdot W\,.
\]
In other words, $\I^{r_1,r_2}(F,K)$ is the space of all processes $u$ that can be written as
\[
  u(t)=u^d(t) + \int_0^tu^s(r)\,\d W(r)\,, \quad t\in[0,T]\,,
\]
for some $u^d\in L^{r_1}_\cP(\Omega; W^{1,r_1}(0,T; E))$ and 
$u_s\in L^{r_2}_\cP(\Omega; L^2(0,T; \cL^2(U,K)))$. Clearly, in such case we have
\[
  u(t)=u^d(0) + \int_0^t\partial_tu^d(s)\,\d s + \int_0^tu^s(r)\,\d W(r)\,, \quad t\in[0,T]\,,
\]
Note that the sum is actually a direct sum, and the representation $u=u^d+u^s\cdot W$
is unique.

\subsection{Assumptions}
\label{ssec:assumptions} 
We now introduce and
comment the assumptions on operators and data. These will be assumed
throughout, without further specific mention. 

\begin{itemize}\setlength\itemsep{1.5em}
\item {\bf Assumption U0:} $u_0 \in V$.

\item{\bf Assumption A:} $A:H\to2^H$ is maximal monotone and there exist
  two constants $c_A,C_A>0$ such that
  \begin{align*}
    \norm{y}_H\leq C_A(1+\norm{x}_H) \qquad&\forall\,x\in H\,,\quad\forall\,y\in A(x)\,,\\
    (y,x)_H\geq c_A\norm{x}_H^2 - c_A^{-1} \qquad&\forall\,x\in
    D(A)\,,\quad\forall\,y\in A(x)\,.
  \end{align*}

\item   {\bf Assumption B1:} $\widehat B:H\to[0,+\infty]$ is convex lower
  semicontinuous with $\widehat B(0)=0$ and $B:=\partial\widehat
  B:H\to 2^H$ is its subdifferential.   We ask that
  \[
  V\subseteq D(\widehat B)\,,
  \]
  and $\widehat B$ is continuous at some point of $V$: this implies
  that $\partial(\widehat B_{|_V}):V\to2^{V^*}$
  coincides with $B$ on $D(B)$, hence $B$ can be also seen as a
  maximal monotone operator from $V$ to $2^{V^*}$. We assume
  that $B$ is single-valued, and that there exist constants $c_B>0$
  and $p\geq2$ such that
  \begin{align*}
    (B(x_1)-B(x_2),x_1-x_2)_H\geq c_B\norm{x_1-x_2}_V^p
    \qquad\forall\,x_1,x_2\in D(B)\,.
  \end{align*}

 \item {\bf Assumption B2:}  $B:V\to V^*$ is G\^ateaux-differentiable with
  \[
  D_{\mathcal G}B\in C^0(V; \cL_w(V,V^*))
  \]
  and there exists a constant $C_B>0$ such that
  \[
  \norm{D_{\mathcal G}B(x)}_{\cL(V,V^*)} \leq C_B\left(1 +
    \norm{x}_V^{p-2}\right) \qquad\forall\,x\in V\,.
  \]

\item   {\bf Assumption F:} $F:[0,T]\times V\to H$ is measurable, and there
  exist a constant $C_F>0$ and $h_F\in L^1(0,T)$ such that
  \begin{align*}
    \norm{F(\cdot, x_1)-F(\cdot,x_2)}^2_H&\leq C_F
    \left(1+\norm{x_1}_V^{p-2}
      + \norm{x_2}_V^{p-2}\right)\norm{x_1-x_2}_V^2 &&\forall\,x_1,x_2\in V\,,\\
    \norm{F(\cdot,x)}_H^2&\leq h_F(\cdot) + C_F\norm{x}_V^p
    &&\forall\,x\in V\,.
  \end{align*}

\item  {\bf Assumption G:} $G:[0,T]\times H\to \cL^2(U,H)$  is
  measurable,  $G([0,T]\times V)\subset\cL(U,V)$, 
  and there exist a constant $C_G>0$ such that
  \begin{align*}
    &\norm{G(\cdot, x_1)-G(\cdot,x_2)}_{\cL^2(U,H)}
    \leq C_G\norm{x_1-x_2}_H &&\forall\,x_1,x_2\in H\,,\\
    &\norm{G(\cdot,x)}_{\cL^2(U,H)}\leq
    C_G\left(1+\norm{x}_V^\nu\right) &&\forall\,x\in V\,,
  \end{align*}
  where $\nu=1$ if $p>2$ and $\nu\in(0,1)$ if $p=2$.  Moreover, 
  we
  assume that there exists
   $L:[0,T]\times V\to \cL^1(H,H)$ and
  $h_G\in L^1(0,T)$ such that 
  \[
  L(\cdot,x)=G(\cdot,x)G(\cdot,x)^*D_{\mathcal G}B(x)
  \qquad\forall\,x\in V\,,
  \]
  and
  \begin{align*}
    &L(t,\cdot)\in C^0(V; \cL^1(H,H)) \qquad\text{for a.e.~$t\in(0,T)$}\,,\\
    &\norm{L(\cdot,x)}_{\cL^1(H,H)}\leq h_G(\cdot) + C_G\norm{x}_V^p
    \qquad\forall\,x\in V\,.
  \end{align*}
\end{itemize}

Let us comment on this last requirement. Note that by the regularity of 
$G$ and $B$, for every~$t\in(0,T)$ and $x\in V$ we have that 
$G(t,x)\in\cL(U,V)$ and $D_{\mathcal G}B(x)\in\cL(V,V^*)$,
so that $G(t,x)G(t,x)^*D_{\mathcal G}B(x)\in\cL(V,V)$. In the 
last condition,
we are assuming that actually such 
operator extends to a trace class operator 
on $H$. Clearly, any such extension is necessarily unique 
due to the density of $V$ in $H$.
This behaviour is very natural and actually occurs in several relevant 
situations: roughly speaking, it means that the composition
with $GG^*$ has a regularizing effect on $D_{\mathcal G}B(x)$,
due for example to suitable compensations of the singular terms.

\subsection{Main results}
\label{ssec:main}
We are now in the position of stating the main results of the paper. 
Our first result concerns existence of 
analytically strong martingale solutions to
\eqref{eq'}. 
\begin{thm}[Existence of martingale solutions]
  \label{thm1}
  Under the assumptions of Subsection \emph{\ref{ssec:assumptions}},
  there exists a probability space $(\hat\Omega,\hat\cF,\hat\P)$,
  a $U$-cylindrical Wiener process $\hat W$ on it, 
  and a triple $(\hat u, \hat u^d,\hat v)$, with
  \begin{align*}
  &\hat u \in L^p_{\hat \cP}(\hat\Omega; C^0([0,T]; H))\cap 
  L^p_{\hat \cP}(\hat\Omega; L^\infty(0,T; V))\,,
  \qquad \hat u^d \in L^2_{\hat \cP}(\hat\Omega; H^1(0,T; H))\,,\\
  &\hat u\in D(B) \quad\text{a.e.~in } \hat\Omega\times(0,T)\,,\qquad
  B(\hat u) \in L^2_{\hat \cP}(\hat\Omega; L^2(0,T; H))\,,\\
  &\hat v \in L^2_{\hat \cP}(\hat\Omega; L^2(0,T, H))\,,
  \end{align*}
  such that
  \begin{align*}
  &\hat u(t) = u_0 + \int_0^t\partial_t\hat u^d(s)\,\d s + \int_0^tG(s,\hat u(s))\,\d \hat W(s)
  &&\forall\,t\in[0,T]\,,\quad \hat\P\text{-a.s.}\,,\\
  &\hat v(t) + B(\hat u(t)) = F(t,\hat u(t)) 
  &&\text{for a.e.~$t\in(0,T)$}\,,\quad \hat\P\text{-a.s.}\,,\\
  &\hat v(t) \in A(\partial_t \hat u^d(t)) 
  &&\text{for a.e.~$t\in(0,T)$}\,,\quad \hat\P\text{-a.s.}
  \end{align*}
  Furthermore, the following energy equality holds:
  \begin{align*}
  &\widehat B(\hat u(t)) + \int_0^t\left(\hat v(s), \partial_t\hat u^d(s)\right)_H\,\d s \\
  &=\widehat B(u_0) + \int_0^t\left(F(s,\hat u(s)), \partial_t\hat u^d(s)\right)_H\,\d s
  +\frac12\int_0^t\operatorname{Tr}\left[L(s,\hat u(s))\right]\,\d s\\
  &\qquad+\int_0^t\left(B(\hat u(s)), G(s,\hat u(s))\,\d W(s)\right)_H
  \qquad\forall\,t\in[0,T]\,,\quad\hat\P\text{-a.s.}
  \end{align*}
\end{thm}

Our second result gives a sufficient condition on the nonlinearities
in order to obtain existence (and uniqueness) of strong solutions
both in the probabilistic and analytical sense.
\begin{thm}[Well posedness in the strong probabilistic sense]
  \label{thm2}
  Under the assumptions of Subsection \emph{\ref{ssec:assumptions}},
  suppose further that 
  $A$ is strongly monotone,
  $B$ is linear self-adjoint,
  $G$ is linear in its second variable, 
  and $h_G=0$.
  Then, 
  there exists a unique triple $(u, u^d, v)$, with
  \begin{align*}
  &u \in L^p_{\cP}(\Omega; C^0([0,T]; H))\cap L^p_\cP(\Omega; L^\infty(0,T; V))\,,
  \qquad u^d \in L^2_\cP(\Omega; H^1(0,T; H))\,,\\
  &B(u) \in L^2_{\cP}(\Omega; L^2(0,T; H))\,,\\
  & v \in L^2_\cP(\Omega; L^2(0,T, H))\,,
  \end{align*}
  such that
  \begin{align*}
  & u(t) = u_0 + \int_0^t\partial_t u^d(s)\,\d s + \int_0^tG(s, u(s))\,\d  W(s)
  &&\forall\,t\in[0,T]\,,\quad \P\text{-a.s.}\,,\\
  & v(t) + B( u(t)) = F(t, u(t)) 
  &&\text{for a.e.~$t\in(0,T)$}\,,\quad \P\text{-a.s.}\,,\\
  & v(t) \in A(\partial_t  u^d(t)) 
  &&\text{for a.e.~$t\in(0,T)$}\,,\quad \P\text{-a.s.}
  \end{align*}
  Furthermore, there exists a constant $K>0$, such that,
  for every set $\{u_{0,i}\}_{i=1,2}\subset V$ of initial data,
  their respective solutions $\{(u_i, u_i^d, v_i)\}_{i=1,2}$ satisfy
  \[
  \norm{u_1-u_2}_{L^\infty(0,T; L^2(\Omega; V))}
  +\norm{u_1^d-u_2^d}_{L^2(\Omega; H^1(0,T; H))}
  \leq K\norm{u_{0,1}-u_{0,2}}_V\,.
  \]
\end{thm}


\section{An extended It\^o's formula}
\label{sec:ito}
In this section we prove a general version of It\^o's formula for $\widehat B$,
under minimal differentiability assumptions, see Proposition~\ref{prop:ito}. 
In particular, we recall that the second derivative of $\widehat B$
is supposed to be defined only in the sense of G\^ateaux in $V$:
hence, $\widehat B$ may even be not twice-differentiable on $H$,
and, in fact, not even well-defined on the whole space $H$.

We find such It\^o's formula very interesting on its own,
as it widely generalizes the classical result to the case where 
the second derivatives are not necessarily well-defined and continuous.

For every $\lambda>0$, we recall that 
the Moreau-Yosida regularization of $\widehat B$,
the resolvent of $B$, and
the Yosida approximation of $B$, are defined respectively as
\begin{align*}
  &\widehat B_\lambda:H\to[0,+\infty)\,, 
  &&\widehat B_\lambda(x):=\inf_{y\in H}\left\{\widehat B(y) + 
  \frac{\norm{x-y}_H^2}{2\lambda}\right\}\,, \quad x\in H\,,\\
  &J_\lambda^B:H\to D(B)\subset V\,, 
  &&J_\lambda^B(x):=(I+\lambda B)^{-1}(x)\,, \quad x\in H\,,\\
  &B_\lambda:H\to H\,, 
  &&B_\lambda(x):=\frac{x - J_\lambda^B(x)}{\lambda}\,, \quad x\in H\,,
\end{align*}
where $I:H \to H$ is the identity in $H$. 
It is well-known that $J_\lambda^B$ is non-expansive on $H$ and that 
$B_\lambda$ is $1/\lambda$-Lipschitz-continuous on $H$.

First of all, we prove some technical properties of the resolvent.
\begin{lem}
  \label{lem:res}
  For every $\lambda>0$, the resolvent $J_\lambda^B$ uniquely extends to 
  a $\frac1{p-1}$-H\"older-continuous 
  G\^ateaux-differentiable operator $J_\lambda^B:V^*\to V$,
  such that 
   $D_{\mathcal G}J_\lambda^B\in C^0(H; \cL_w(H,H))$. 
  Moreover, as $\lambda\searrow0$ it holds that
  \begin{align*}
    J_\lambda^B(x)\to x \qquad&\text{in } H &&\forall\,x\in H\,,\\
    J_\lambda^B(x)\wto x \qquad&\text{in } V &&\forall\,x\in V\,,\\
    D_{\mathcal G}J_\lambda^B(x)\to I \qquad&\text{in } \cL_s(H,H)
     &&\forall\,x\in H\,. 
  \end{align*}
  In particular, there exists a constant $M>0$, independent of $\lambda$, such that
  \[
  \norm{J_\lambda(x)}_V\leq M\left(1+\norm{x}_V\right) \qquad\forall\,x\in V\,.
  \]
\end{lem}
\begin{proof}
  {\sc Step 1.} Let $\lambda>0$ be fixed. The operator
  \[
  I + \lambda B:V\to V^*
  \]
  is maximal monotone and coercive, hence surjective.
  Its inverse $(I + \lambda B)^{-1}:V^*\to V$ is well-defined,
  and coincides with $J_\lambda^B$ on $H$. 
  In the sequel, we will use the same symbol $J_\lambda^B$
  to denote $(I + \lambda B)^{-1}:V^*\to V$.
  Note that actually $J_\lambda^B:V^*\to V$
  is $\frac1{p-1}$-H\"older-continuous by the strong monotonicity of $B$.
  Indeed, for every $x_1,x_2\in V^*$,
  it is immediate to see that 
  \[
  \norm{J_\lambda^B(x_1)-J_\lambda^B(x_2)}_H^2
  +\lambda c_B\norm{J_\lambda^B(x_1)-J_\lambda^B(x_2)}_V^p
  \leq\ip{x_1-x_2}{J_\lambda^B(x_1)-J_\lambda^B(x_2)}_{V^*,V}\,,
  \]
  from which it follows that 
  \[
  \norm{J_\lambda^B(x_1)-J_\lambda^B(x_2)}_V\leq
  \left(\frac{1}{\lambda c_B}\right)^{\frac1{p-1}}\norm{x_1-x_2}_{V^*}^{\frac1{p-1}}\,.
  \]
  Hence, $(I + \lambda B)^{-1}:V^*\to V$ is H\"older-continuous, 
  and the extension of $J_\lambda^B$ to $V^*$ is unique.
  Now, by {\bf B1--B2} we know that $I+\lambda B:V\to V^*$ is G\^ateaux-differentiable
 with differential given by 
  \[
  T_\lambda(x):=I + \lambda D_{\mathcal G}B(x) \in \cL(V,V^*)\,, \qquad x\in V\,.
  \]
  It follows that for every $x\in V$,
  the linear continuous operator $T_\lambda(x):V\to V^*$ is monotone 
  (hence maximal monotone),
  injective, and coercive on $V$ (due to the strong monotonicity of $B$): 
  this implies that $T_\lambda(x)$ is an isomorphism for all $x\in V$.
  Consequently, as we have already proved that 
  $I+\lambda B:V\to V^*$ is invertible with H\"older-continuous inverse,
  we deduce that $J_\lambda^B:V^*\to V$ is G\^ateaux-differentiable with 
  \[
  D_{\mathcal G}J_\lambda^B(x)=
  T_\lambda(J_\lambda^B(x))^{-1}\in\cL(V^*,V)\,, \qquad x\in V^*\,.
  \]
  Clearly, this implies in particular that $J_\lambda^B:H\to H$ is G\^ateaux-differentiable,
  and
  \beq\label{eq_res}
  D_{\mathcal G}J_\lambda^B(y)h + \lambda 
  D_{\mathcal G}B(J_\lambda^B(y))D_{\mathcal G}J_\lambda^B(y)h 
  = h \qquad\forall\,h,y\in H\,.
  \eeq
  As a by-product, we deduce also that $B_\lambda:H\to H$ is G\^ateaux-differentiable with
  \[
  D_{\mathcal G}B_\lambda(y)=D_{\mathcal G}B(J_\lambda^B(y))D_{\mathcal G}J_\lambda^B(y)
  \qquad\forall\,y\in H\,.
  \]
  Let us show that actually  $D_{\mathcal G}J_\lambda^B\in C^0(H; \cL_w(H,H))$. 
  To this end, let $(x_j)_j\subset H$ and $x\in H$ such that 
  $x_j\to x$ in $H$ as $j\to\infty$. For what we have just proved, 
  it holds that $J_\lambda^B(x_j)\to J_\lambda^B(x)$ in $V$.
 Taking $y=x_j$ and $h\in H$ in \eqref{eq_res},
 and testing by $D_{\mathcal G}J_\lambda^B(x_j)h$ yields,
 by the monotonicity of $B$ and the Young inequality,
 \[
 \frac12\norm{D_{\mathcal G}J_\lambda^B(x_j)h}_H^2
  \leq\frac12\norm{h}_H^2 \qquad\forall\,j\in\enne\,.
 \] 
  Hence, there exists $I_\lambda(x,h)\in H$ such that, as $j\to\infty$,
  \[
  D_{\mathcal G}J_\lambda^B(x_j)h \wto I_\lambda(x,h) \qquad\text{in } H\,,\quad\forall\,h\in H\,.
  \]
  Let now $k\in H$ be arbitrary, and
  consider the double sequence 
  \[
  a_{j,i}:=\left(\frac{J_\lambda^B(x_j+ h/i)-J_\lambda^B(x_j)}{1/i}, k\right)_H\,, \qquad (j,i)\in\enne^2\,.
  \]
  Since $J_\lambda^B$ is non-expansive on $H$, the sequence $\{a_{j,i}\}_{j,i}$
  is uniformly bounded in $(j,i)$. Hence, 
  there is $a_\lambda\in\erre$ such that, possibly on a 
  non-relabelled subsequence of $(j,i)$, it holds
  \[
  \lim_{j,i\to\infty}a_{j,i}= a \,.
  \]
  Now, note also that by continuity and G\^ateaux-differentiability of $J_\lambda^B$ we have 
  \[
  \lim_{j\to\infty}a_{j,i}=\left(\frac{J_\lambda^B(x+ h/i)-J_\lambda^B(x)}{1/i}, k\right)_H
  \qquad\forall\,i\in\enne\,,
  \]
  and
  \[
  \lim_{i\to\infty}a_{j,i}=\left(D_{\mathcal G}J_\lambda^B(x_j)h,k\right)_H \qquad\forall\,j\in\enne\,.
  \]
  Consequently, since the partial limits are finite and the double sequence converges also 
  jointly, it holds that 
  \[
  a=\lim_{j,i\to\infty}a_{j,i} = \lim_{j\to\infty}\lim_{i\to\infty}a_{j,i} 
  =\lim_{i\to\infty}\lim_{j\to\infty}a_{j,i}\,.
  \]
  Putting this information together, we deduce that 
  \[
  \lim_{j\to\infty}\left(D_{\mathcal G}J_\lambda^B(x_j)h,k\right)_H = 
  \left(D_{\mathcal G}J_\lambda^B(x)h,k\right)_H = a \qquad\forall\,h,k\in H\,.
  \]
  Since the limit $a$ is independent of the specific subsequence, this 
  argument can be replicated on all subsequences of $(j,i)$, and the convergences 
  holds actually along the entire sequence. By the arbitrariness of $k\in H$,
  it follows that 
  $I_\lambda(x,h)=D_{\mathcal G}J_\lambda^B(x)h$ for all $h\in H$.
   This shows that $D_{\mathcal G}J_\lambda^B\in C^0(H; \cL_w(H,H))$, as required.\\
  {\sc Step 2.}
  Let us show the behaviour as $\lambda\searrow0$: let $x\in H$ be fixed. 
  First of all, recalling that 
  \[
  J_\lambda^B(x) + \lambda B(J_\lambda^B(x)) = x\,,
  \]
  testing by $J_\lambda^B(x)$ and using the strong monotonicity of
  $B$, the fact that $0 \in B(0)$, 
  and the Young inequality we get
  \beq\label{aux_res1}
  \frac12\norm{J_\lambda^B(x)}_H^2 + \lambda c_B\norm{J_\lambda^B(x)}_V^p
  \leq \frac12\norm{x}_H^2\,.
  \eeq
  It follows that $(J_\lambda^B(x))_\lambda$ is relatively weakly compact  in $H$.
  Moreover, by assumption {\bf B2} and \eqref{aux_res1} we have that,
  for a positive constant $M$ independent of $\lambda$,
  \[
  \lambda\norm{B(J_\lambda^B(x))}_{V^*}\leq 
  M\lambda\left(1+\norm{J_\lambda^B(x)}_V^{p-1}\right)
  \leq M\lambda + M\lambda(2c_B\lambda)^{-\frac{p-1}{p}}\norm{x}_H^{2\frac{p-1}{p}} \to 0\,.
  \]
  It follows then that $J_\lambda^B(x)\wto x$ in $H$, hence also,
  again by a classical $\limsup$-argument, that 
  $J_\lambda^B(x)\to x$ in $H$.
  Secondly, taking $h\in H$ and $y=x$
  in \eqref{eq_res}, testing 
  by $D_{\mathcal G}J_\lambda^B(x)h$, using the monotonicity of $B$
  and the Young inequality  yield
  \beq\label{aux_res2}
  \frac12\norm{D_{\mathcal G}J_\lambda^B(x)h}_H^2
  \leq\frac12\norm{h}_H^2\,.
  \eeq
  Hence, there exists $\tilde I(x,h)\in H$ such that, as $\lambda\searrow0$,
  \[
  D_{\mathcal G}J_\lambda^B(x)h \wto \tilde I(x,h) \qquad\forall\,h\in H\,.
  \]
  We show that $\tilde I(x,h)=h$ using again 
  an argument with double sequences. Let $k\in H$ be arbitrary
  and consider 
  \[
  b_{j,i}:=\left(\frac{J_{\lambda_j}^B(x+ h/i)
  -J_{\lambda_j}^B(x)}{1/i}, k\right)_H\,, \qquad (j,i)\in\enne^2\,,
  \]
  where $(\lambda_j)_j$ is an arbitrary infinitesimal real sequence.
  Since $J_\lambda^B$ is non-expansive on $H$, the sequence $\{b_{j,i}\}_{j,i}$
  is uniformly bounded in $(j,i)$, and 
  there is $b\in\erre$ such that, possibly on a 
  non-relabelled subsequence of $(j,i)$,
  \[
  \lim_{j,i\to\infty}b_{j,i}= b \,.
  \]
  As before, by the convergence of $J_\lambda$ as $\lambda\searrow0$
  and the G\^ateaux-differentiability of $J_\lambda^B$ we have 
  \[
  \lim_{j\to\infty}b_{j,i}=\left(h, k\right)_H
  \qquad\forall\,i\in\enne\,,
  \]
  and
  \[
  \lim_{i\to\infty}b_{j,i}=\left(D_{\mathcal G}J_{\lambda_j}^B(x)h,k\right)_H \qquad\forall\,j\in\enne\,.
  \]
  Consequently, since the partial limits are finite and the double sequence converges also 
  jointly, the double limits can be computed in any order, so that it holds that 
  \[
  \lim_{j\to\infty}\left(D_{\mathcal G}J_{\lambda_j}^B(x)h,k\right)_H = 
  \left(h,k\right)_H \qquad\forall\,h,k\in H\,.
  \]
  Since the limit is independent of the specific subsequence, this 
  argument can be replicated on all subsequences of $(j,i)$, and the convergences 
  holds actually along the entire sequence. By the arbitrariness of $k\in H$,
  it follows that 
  $\tilde I(x,h)=h$ for all $h\in H$.
  Furthermore, from \eqref{aux_res2} we also have 
  \[
  \limsup_{\lambda\searrow0}\norm{D_{\mathcal G}J_\lambda^B(x)h}_H^2
  \leq\norm{h}_H^2\,,
  \]
  hence actually it holds that $D_{\mathcal G}J_\lambda^B(x)h\to h$ in $H$.
  This shows that $D_{\mathcal G}J_\lambda\to I$ in $\cL_s(H,H)$, as required.\\
  {\sc Step 3.}
  Lastly, let us suppose that $x\in V$: from the relation
  \[
    J_\lambda^B(x) + \lambda B(J_\lambda^B(x)) = x\,,
  \]
  testing by $B(J_\lambda^B(x))$ and using {\bf B1--B2} together with
  the Young inequality we get 
  \begin{align*}
  c_B\norm{J_\lambda^B(x)}_V^p + \lambda\norm{B(J_\lambda^B(x))}_H^2
  &=\left(B(J_\lambda^B(x)), x\right)_H\leq
  C_B\norm{x}_V\left(1 + \norm{J_\lambda^B(x)}_V^{p-1}\right)\\
  &\leq C_B\norm{x}_V + 
  \frac{1}{p}C_B^pc_B^{1-p}\norm{x}_V^{p} +
  \frac{p-1}pc_B\norm{J_\lambda^B(x)}^p\,,
  \end{align*}
  from which
  \beq
  \label{aux_res3}
  \frac{c_B}p\norm{J_\lambda^B(x)}_V^p +  \lambda\norm{B(J_\lambda^B(x))}_H^2
  \leq C_B\norm{x}_V + \frac{1}{p}C_B^pc_B^{1-p}\norm{x}_V^{p}\,.
  \eeq
  It follows that $J_\lambda^B(x)\wto x$ in $V$ as $\lambda\searrow0$, and the proof is concluded.
\end{proof}

Lemma~\ref{lem:res} implies in particular that the Yosida approximation 
$B_\lambda:H\to H$ is G\^ateaux-differentiable, but not necessarily in 
the sense of Fr\'echet.
In the next lemma we show that it is possible to write It\^o's formula 
for $\widehat B_\lambda$, even if its second derivative is not 
well-defined everywhere in the sense of Fr\'echet.
\begin{lem}
  \label{lem:ito_lam}
   Let the processes $v$, $f$, and $C$ satisfy  
  \begin{align*}
  v\in L^2_\cP(\Omega; C^0([0,T]; H))\,, \qquad
  f \in L^2_\cP(\Omega; L^2(0,T; H))\,, \qquad
  C\in L^2_\cP(\Omega; L^2(0,T; \cL^2(U,H)))\,,
  \end{align*}
  and
  \[
  v(t) +\int_0^tf(s)\,\d s = v(0) + \int_0^tC(s)\,\d W(s) \qquad\forall\,t\in[0,T]\,,\quad\P\text{-a.s.}
  \]
  Then, for all $\lambda>0$ it holds that
  \begin{align*}
  \widehat B_\lambda(v(t)) + \int_0^t(f(s), B_\lambda(v(s)))_H\,\d s &= 
  \widehat B_\lambda(v(0)) 
  + \frac12\int_0^t\operatorname{Tr}\left[
  C^*(s)D_{\mathcal G}B_\lambda(v(s))C(s)\right]\,\d s\\
  &+\int_0^t(B_\lambda(v(s)), C(s)\,\d W(s))_H \qquad
  \forall\,t\in[0,T]\,,\quad\P\text{-a.s.}
  \end{align*}
\end{lem}
\begin{proof}
  Let $\{e_k\}_{k\in\enne}$ be a complete orthonormal 
  system of $H$ included in $V$ and set 
  \[
  H_n:=\operatorname{span}\{e_0\,,\ldots\,,e_{n-1}\}\,, \quad n\in\enne\,, \qquad
  H_\infty:=\bigcup_{n\in\enne} H_n\,.
  \]
  Clearly, $H_n$ is a closed linear subspace of $H$ for all $n\in\enne$,
  and $H_\infty$ is dense in $H$.
  For every $n\in\enne$, it is well-defined the 
  orthogonal projection $P_n:H\to H_n$ on $H_n$, with respect to the scalar 
  product of $H$. Clearly, $P_n$ is linear, $1$-Lipschitz continuous, and
  \[
  P_nx\to x \quad\text{in $H$} \qquad\forall\,x\in\text{$H$}\,.
  \]
  Setting now $v_n:=P_nv$, $f_n:=P_nf$, and $C_n:=P_nC$, we have that,
  for every $n\in\enne$,
  \begin{equation}
  v_n(t) +\int_0^tf_n(s)\,\d s = v_n(0) + \int_0^tC_n(s)\,\d W(s) 
  \quad\text{in } H_n\,,\qquad\forall\,t\in[0,T]\,,\quad\P\text{-a.s.}\label{sette}
  \end{equation}
  Now, it is immediate to see that the restriction 
  $\widehat B_{\lambda|_{H_n}}\in C^1(H_n)$ and 
  its differential is given by 
  \[
  D(\widehat B_{\lambda|_{H_n}})=P_n\circ B_{\lambda|_{H_n}}:H_n\to H_n\,.
  \]
  Furthermore, from Lemma~\ref{lem:res} we have that $B_\lambda:H\to H$
  is G\^ateaux-differentiable, which yields directly that 
  $P_n\circ B_{\lambda|_{H_n}}:H_n\to H_n$ is G\^ateaux-diifferentiable in turn with 
  \[
  D_{\mathcal G}(P_n\circ B_{\lambda|_{H_n}})(x)h =
  P_nD_{\mathcal G}B_\lambda(x)h\,, \qquad x\in H_n\,, \quad h\in H_n\,.
  \]
  Let us show that actually $D^2_{\mathcal G}\widehat B_\lambda \in C^0(H_n; \cL(H_n,H_n))$.
  To this end, let $\{x_j\}_j\subset H_n$ and $x\in H_n$ such that
  $x_j\to x$ in $H_n$ as $j\to\infty$. Then we have 
  \begin{align*}
    \norm{D^2_{\mathcal G}\widehat B_\lambda(x_j)-
    D^2_{\mathcal G}\widehat B_\lambda(x)}_{\cL(H_n,H_n)}
    &=
    \sup_{\norm{h}_{H_n}\leq1}\left\{
    \norm{D^2_{\mathcal G}\widehat B_\lambda(x_j)h-
    D^2_{\mathcal G}\widehat B_\lambda(x)h}_{H_n}
    \right\}\\
    &=\sup_{\norm{h}_{H_n}\leq1}
    \sup_{\norm{k}_{H_n}\leq1}
    \left\{
    \left((D_{\mathcal G}B_\lambda(x_j)-
    D_{\mathcal G} B_\lambda(x))h, k\right)_{H_n}
    \right\}\,.
  \end{align*}
  Now, for any such arbitrary $h,k\in H_n$, we can write
  \[
  h=\sum_{i=1}^na^h_i e_i\,, \qquad 
  k=\sum_{i=1}^na^k_i e_i\,,
  \]
  where
  \[
  a_i^h:=(h,e_i)_H\,, \qquad a_i^k:=(k,e_i)_H\,, \qquad i=1,\ldots,n\,.
  \]
  Clearly, $\norm{h}_{H_n}\leq1$ and $\norm{k}_{H_n}\leq 1$ if and only if 
  \[
  \sum_{i=1}^n|a_i^h|^2\leq 1\,, \qquad \sum_{i=1}^n|a_i^k|^2\leq 1\,.
  \]
  Consequently, setting $B_n$ as the closed unit ball in $\erre^n$,
  we have $a^h:=(a_1^h,\ldots,a_n^h)\in B_n$ and 
  $a^k:=(a_1^k,\ldots,a_n^k)\in B_n$. We deduce that 
  \begin{align*}
    \norm{D^2_{\mathcal G}\widehat B_\lambda(x_j)-
    D^2_{\mathcal G}\widehat B_\lambda(x)}_{\cL(H_n,H_n)}
    &=
    \sup_{a^h,a^k\in B_n}
    \left\{\sum_{i,l=1}^n a_i^ha_l^k
    \left((D_{\mathcal G}B_\lambda(x_j)-
    D_{\mathcal G} B_\lambda(x))e_i, e_l\right)_{H_n}
    \right\}\\
    &\leq\sum_{i,l=1}^n\left|
    \left((D_{\mathcal G} B_\lambda(x_j)-
    D_{\mathcal G}B_\lambda(x))e_i, e_l\right)_{H}\right|\,,
  \end{align*}
  where the right-hand side converges to $0$ when $j\to\infty$
  because $D_{\mathcal G} B_\lambda\in C^0(H; \cL_w(H,H))$ by 
  Lemma~\ref{lem:res}.
  Hence, we have proved that $\widehat B_{\lambda|_{H_n}}
   \in C^2(H_n)$ for every $n\in\enne$, and
  the finite dimensional It\^o formula yields then
  \begin{align*}
  \widehat B_\lambda(v_n(t)) + \int_0^t(f_n(s), B_\lambda(v_n(s)))_H\,\d s &= 
  \widehat B_\lambda(v_n(0)) 
  + \frac12\int_0^t\operatorname{Tr}
  \left[C_n^*(s)D_{\mathcal G}B_\lambda(v_n(s))C_n(s)\right]\,\d s\\
  &+\int_0^t(B_\lambda(v_n(s)), C_n(s)\,\d W(s))_H \qquad
  \forall\,t\in[0,T]\,,\quad\P\text{-a.s.}
  \end{align*}
  We want to pass now to the limit as $n\to\infty$. To this end, 
  note that by the properties of $P_n$ we have that 
  \begin{align*}
    v_n\to v \qquad&\text{in } L^2_\cP(\Omega; L^\ell(0,T; H)) \quad\forall\,\ell\geq1\,,\\
    v_n(t)\to v(t) \qquad&\text{in } L^2(\Omega,\cF_t; H)\,, \quad\forall\,t\in[0,T]\,,\\
    f_n\to f \qquad&\text{in } L^2_\cP(\Omega; L^2(0,T; H))\,,\\
    C_n\to C \qquad&\text{in } L^2_\cP(\Omega; L^2(0,T; \cL^2(U,H)))\,.
  \end{align*}
  Moreover, by the Lipschitz-continuity of $B_\lambda$,
  the quadratic growth of $\widehat B_\lambda$, and
  the Vitali Convergence Theorem ensure that
  \begin{align*}
   B_\lambda(v_n)\to B_\lambda(v) \qquad&\text{in } 
   L^2_\cP(\Omega; L^\ell(0,T; H)) \quad\forall\,\ell\geq1\,,\\
   \widehat B_\lambda(v_n(t))\to \widehat B_\lambda(v(t))
   \qquad&\text{in } L^1(\Omega,\cF_t)\,, \quad\forall\,t\in[0,T]\,.
  \end{align*}
  Furthermore, Lemma~\ref{lem:res} and the definition of Yosida approximation
  imply that 
  \[
    D_{\mathcal G}B_\lambda \in C^0(H; \cL_w(H,H))\,,
  \]
  so that by the Dominated Convergence Theorem we also infer that 
  \[
  \operatorname{Tr}
  \left[C_n^*D_{\mathcal G}B_\lambda(v_n)C_n\right]
  \to\operatorname{Tr}
  \left[C^*D_{\mathcal G}B_\lambda(v)C\right] \qquad\text{in } L^1_{\cP}(\Omega; L^1(0,T))\,.
  \]
  Eventually, noting that by the dominated convergence theorem
  \[
  B_\lambda(v_n(s))C_n\to B_\lambda(v_n)C \qquad\text{in } 
  L^2_\cP(\Omega; L^2(0,T; \cL^2(U,\erre)))\,,
  \]
  the Burkholder-Davis-Gundy inequality yields
  \[
  \int_0^\cdot(B_\lambda(v_n(s)), C_n(s)\,\d W(s))_H \to 
  \int_0^\cdot(B_\lambda(v(s)), C(s)\,\d W(s))_H \qquad\text{in } L^2_\cP(\Omega; C^0([0,T]))\,.
  \]
  Taking this information into account, 
we can let $n\to\infty$ in \eqref{sette}, and deduce that It\^o's formula holds for 
  a suitable $\d\P\otimes\d t$-version of $v$, denoted by the same symbol 
  for brevity of notation.
\end{proof}

We are now ready to prove the most general version of It\^o formula for $\widehat B$.
Let us stress again that this is not obvious at all, as the second derivative 
of $\widehat B$ is not everywhere defined and is intended only in the sense of G\^ateaux.
\begin{prop}[It\^o formula for $\widehat B$]
  \label{prop:ito}
  Let the processes $v$ and $f$ satisfy  
  \begin{align*}
  &v\in L^2_\cP(\Omega; C^0([0,T]; H))\,, \qquad
  \widehat B(v(t)) \in L^1(\Omega,\cF_t) \quad\forall\,t\in[0,T]\,,\\
  &v\in D(B)\quad\text{a.e.~in } \Omega\times(0,T)\,, \qquad
  B(v)\in L^2_\cP(\Omega; L^2(0,T; H))\,,\\
  &f \in L^2_\cP(\Omega; L^2(0,T; H))\,, 
  \end{align*}
  and
  \[
  v(t) +\int_0^tf(s)\,\d s = v(0) + \int_0^tG(s,v(s))\,\d W(s) \qquad\forall\,t\in[0,T]\,,\quad\P\text{-a.s.}
  \]
  Then, it holds that
  \begin{align*}
  \widehat B(v(t)) + \int_0^t(f(s), B(v(s)))_H\,\d s &= 
  \widehat B(v(0)) + \frac12\int_0^t\operatorname{Tr}
  \left[L(s,v(s))\right]\,\d s\\
  &+\int_0^t\big(B(v(s)), G(s,v(s))\,\d W(s)\big)_H \qquad
  \forall\,t\in[0,T]\,,\quad\P\text{-a.s.}
  \end{align*}
\end{prop}
\begin{proof}
For every $\lambda>0$, there exists a unique $v_\lambda\in L^2_\cP(\Omega; C^0([0,T]; H))$
such that 
\[
  v_\lambda(t) +\int_0^tf(s)\,\d s 
  = v(0) + \int_0^tG(s,J_\lambda^B
  (v_\lambda(s)))\,\d W(s) \qquad\forall\,t\in[0,T]\,,\quad\P\text{-a.s.}
\]
Hence, it holds that 
\[
  (v-v_\lambda)(t)
  = 
  \int_0^t \left(G(s,v(s)) - G(s,J_\lambda^B(v_\lambda(s)))\right)\,\d W(s) 
  \qquad\forall\,t\in[0,T]\,,\quad\P\text{-a.s.}
\]
It\^o's isometry, the Lipschitz-continuity of $G$,
and the properties of $J_\lambda^B$ and $B_\lambda$ yield then, for every $t\in[0,T]$,
\begin{align*}
  \E\norm{(v-v_\lambda)(t)}_H^2
  &=\E\int_0^t\norm{G(s,v(s)) - G(s,J_\lambda^B(v(s)))}_{\cL^2(U,H)}^2\,\d s\\
  &\leq C_G\int_0^t\E\norm{(v-J_\lambda^B(v_\lambda))(s)}_H^2\,\d s\\
  &\leq 2C_G\int_0^t\E\norm{(v - J_\lambda^B(v))(s)}_H^2
  +2C_G\int_0^t\E\norm{(J_\lambda^B(v)-J_\lambda^B(v_\lambda))(s)}_H^2\\
  &\leq 2C_G\lambda^2\int_0^t\E\norm{B_\lambda(v(s))}_H^2
  +2C_G\int_0^t\E\norm{(v-v_\lambda)(s)}_H^2\\
  &\leq 2C_G\lambda^2 \norm{B(v)}_{L^2_\cP(\Omega; L^2(0,T; H))}^2
  +2C_G\int_0^t\E\norm{(v-v_\lambda)(s)}_H^2\,.
\end{align*}
The Gronwall lemma implies that there exists $M>0$, 
independent of $\lambda$, such that 
\[
  \norm{v_\lambda-v}_{C^0([0,T]; L^2(\Omega; H))} \leq M\lambda\,.
\]
In particular, it holds that $v_\lambda\to v$ in $C^0([0,T]; L^2(\Omega; H))$.
A classical argument involving the Burkholder-David-Gundy 
inequality allows then to show that 
\beq\label{conv1_ito}
  v_\lambda \to v \qquad\text{in } L^2_\cP(\Omega; C^0([0,T]; H))\,.
\eeq
Moreover, since $B_\lambda$ is $\frac1\lambda$-Lipschitz-continuous on $H$,
by the estimate just proved we also have that 
\begin{align*}
  \norm{B_\lambda(v_\lambda)}_{L^2_\cP(\Omega; L^2(0,T; H))}&\leq
  \norm{B_\lambda(v)}_{L^2_\cP(\Omega; L^2(0,T; H))} + 
  \norm{B_\lambda(v_\lambda) - B_\lambda(v)}_{L^2_\cP(\Omega; L^2(0,T; H))}\\
  &\leq \norm{B(v)}_{L^2_\cP(\Omega; L^2(0,T; H))} + 
  \frac1\lambda\norm{v_\lambda - v}_{L^2_\cP(\Omega; L^2(0,T; H))}\\
  &\leq \norm{B(v)}_{L^2_\cP(\Omega; L^2(0,T; H))} + M\,.
\end{align*}
It follows that $(B_\lambda(v_\lambda))_\lambda$ is uniformly bounded in 
$L^2_\cP(\Omega; L^2(0,T; H))$:
as we already know that $v_\lambda\to v$ in $L^2_\cP(\Omega;L^2(0,T; H))$,
by strong-weak closure of $B$ we also have that 
\beq\label{conv2_ito}
  B_\lambda(v_\lambda) \wto B(v) \qquad\text{in } L^2_\cP(\Omega; L^2(0,T; H))\,.
\eeq
As a consequence, by the strong monotonicity of $B$ in {\bf B1}, we get
\[
  c_B\norm{J_\lambda^B(v_\lambda)-v}_{L^p_\cP(\Omega; L^p(0,T; V))}^p
  \leq \E\int_0^T\left(B_\lambda(v_\lambda(s)) - B(v(s)), 
  J_\lambda^B(v_\lambda(s))-v(s)\right)_H\,\d s\to 0\,,
\]
so that 
\beq\label{conv3_ito}
  J_\lambda^B(v_\lambda) \to v \qquad\text{in } L^p_\cP(\Omega; L^p(0,T; V))\,.
\eeq
Now, by Lemma~\ref{lem:ito_lam}, for every $\lambda>0$ we have,
\begin{align}
  \nonumber
  &\widehat B_\lambda(v_\lambda(t)) + 
  \int_0^t(f(s), B_\lambda(v_\lambda(s)))_H\,\d s\\
   \nonumber&= 
  \widehat B_\lambda(v(0)) 
  + \frac12\int_0^t\operatorname{Tr}
  \left[G^*(s,J_\lambda^B(v_\lambda(s)))
  D_{\mathcal G}B_\lambda(v_\lambda(s))G(s,J_\lambda^B(v_\lambda(s)))\right]\,\d s\\
  &\label{ito_lam}\qquad
   +\int_0^t(B_\lambda(v_\lambda(s)), G(s,J_\lambda^B(v_\lambda(s)))\,\d W(s))_H
   \qquad\forall\,t\in[0,T]\,,\quad\P\text{-a.s.}
\end{align}
By definition of $\widehat B_\lambda$ and $B_\lambda$, we have that 
\begin{align*}
  \widehat B_\lambda(v_\lambda) &= \widehat B(J_\lambda(v_\lambda))
  +\frac\lambda2\norm{B_\lambda(v_\lambda)}_H^2
  =\left(B_\lambda(v_\lambda), J_\lambda(v_\lambda)\right)_H
  -\widehat B^*(B_\lambda(v_\lambda)) +\frac\lambda2\norm{B_\lambda(v_\lambda)}_H^2\,,
\end{align*}
from which, recalling \eqref{conv2_ito} and that $\widehat B^*$
is weakly lower semicontinuous on $H$,
\begin{align*}
  &\limsup_{\lambda\searrow0}
  \E\int_0^T\widehat B_\lambda(v_\lambda(s))\,\d s\\
  &\leq\lim_{\lambda\searrow0}
  \E\int_0^T\left(B_\lambda(v_\lambda(s)),J_\lambda(v_\lambda(s))\right)_H\,\d s
  -\liminf_{\lambda\searrow0} 
  \E\int_0^T\widehat B^*(B_\lambda(v_\lambda(s)))\,\d s\\
  &\leq\E\int_0^T\left(B(v(s)),v(s))\right)_H\,\d s
  -\E\int_0^T\widehat B^*(B(v(s)))\,\d s = \E\int_0^T\widehat B(v(s))\,\d s\,.
\end{align*}
As the $\liminf$ inequality is immediate due to
the lower semicontinuity of $\widehat B$, we infer that 
\beq\label{conv4_ito}
  \widehat B_\lambda(v_\lambda) \to \widehat B(v) \qquad\text{in } L^1_\cP(\Omega; L^1(0,T))\,.
\eeq
Furthermore, by the Lipschitz continuity of $G$ we get 
\begin{align*}
  &\norm{B_\lambda(v_\lambda)G(\cdot, J_\lambda^B(v_\lambda)) 
  - B(v)G(\cdot,v)}_{\cL^2(U,\erre)}\\
  &\leq\norm{B_\lambda(v_\lambda)
  (G(\cdot, J_\lambda^B(v_\lambda)) -G(\cdot,v))}_{\cL^2(U,\erre)}
  +\norm{(B_\lambda(v_\lambda) - B(v))G(\cdot,v)}_{\cL^2(U,\erre)}\\
  &\leq C_G\norm{B_\lambda(v_\lambda)}_H\norm{J_\lambda^B(v_\lambda)-v}_H
  +\norm{(B_\lambda(v_\lambda) - B(v))G(\cdot,v)}_{\cL^2(U,\erre)}\,,
\end{align*}
so that from \eqref{conv1_ito}--\eqref{conv2_ito} and the Dominated Convergence Theorem 
we deduce that 
\[
  B_\lambda(v_\lambda)G(\cdot,J_\lambda^B(v_\lambda)) \to B(v)G(\cdot,v) \qquad\text{in }
  L^2_{\cP}(\Omega; L^2(0,T; \cL^2(U,\erre)))\,.
\]
It follows, thanks to the Burkholder-Davis-Gundy inequality, that
\begin{align}
  \nonumber
  &\int_0^\cdot(B_\lambda(v_\lambda(s)), G(s,J_\lambda^B(v_\lambda(s)))\,\d W(s))_H\\
  &\qquad\label{conv5_ito}\to
  \int_0^\cdot(B(v(s)), G(s,v(s))\,\d W(s))_H \quad\text{in } L^2_\cP(\Omega; C^0([0,T]))\,.
\end{align}
It remains to show that we can let $\lambda\searrow0$ in the trace term. To this end, 
note that by Lemma~\ref{lem:res} and the definition of $L$ in {\bf G} we have 
\begin{align*}
  \operatorname{Tr}
  \left[G^*(\cdot,J_\lambda^B(v_\lambda))
  D_{\mathcal G}B_\lambda(v_\lambda)G(\cdot,J_\lambda^B(v_\lambda))\right]
  &=\operatorname{Tr}
  \left[G(\cdot,J_\lambda^B(v_\lambda))G^*(\cdot,J_\lambda^B(v_\lambda))
  D_{\mathcal G}B_\lambda(v_\lambda)\right]\\
  &=\operatorname{Tr}
  \left[L(\cdot,J_\lambda^B(v_\lambda))
  D_{\mathcal G}J^B_\lambda(v_\lambda)\right]\,,
\end{align*}
where, by the strong convergence \eqref{conv3_ito}, the continuity of $L$ in
assumption {\bf G}, and again Lemma~\ref{lem:res},
\begin{align*}
  L(\cdot,J_\lambda^B(v_\lambda)) \to L(\cdot,v) \qquad&\text{in } \cL^1(H,H)\,,
  \quad\text{a.e.~in } \Omega\times(0,T)\,,\\
  D_{\mathcal G}J^B_\lambda(v_\lambda) \to I \qquad&\text{in } \cL_s(H,H)\,,
  \quad\text{a.e.~in } \Omega\times(0,T)\,.
\end{align*}
It follows then that 
\[
  \operatorname{Tr}
  \left[L(\cdot,J_\lambda^B(v_\lambda))
  D_{\mathcal G}J^B_\lambda(v_\lambda)\right] \to 
  \operatorname{Tr}
  \left[L(\cdot,v)\right] \qquad\text{a.e.~in } \Omega\times(0,T)\,.
\]
Moreover, by assumption {\bf G} we have that 
\begin{align*}
  |\operatorname{Tr}
  \left[L(\cdot,J_\lambda^B(v_\lambda))
  D_{\mathcal G}J^B_\lambda(v_\lambda)\right]|
  &\leq\norm{L(\cdot,J_\lambda^B(v_\lambda))}_{\cL^1(H,H)}
  \norm{D_{\mathcal G}J^B_\lambda(v_\lambda)}_{\cL(H,H)}\\
  &\leq
  \norm{L(\cdot,J_\lambda^B(v_\lambda))}_{\cL^1(H,H)}\\
  &\leq h_G(\cdot) + C_G\norm{J_\lambda^B(v_\lambda)}_V^p\,.
\end{align*}
From \eqref{conv3_ito} the right-hand converges in $L^1_\cP(\Omega; L^1(0,T))$,
hence it is uniformly integrable, and so is by comparison the left-hand side.
Putting all this information together, the
Vitali Convergence Theorem yields then
\beq\label{conv6_ito}
 \operatorname{Tr}
  \left[G^*(\cdot,J_\lambda^B(v_\lambda))
  D_{\mathcal G}B_\lambda(v_\lambda)G(\cdot,J_\lambda^B(v_\lambda))\right]  \to 
  \operatorname{Tr}
  \left[L(\cdot,v)\right] \qquad\text{in } L^1_{\cP}(\Omega; L^1(0,T))\,.
\eeq
Letting then $\lambda\searrow0$ in \eqref{ito_lam}
and using the convergences \eqref{conv1_ito}--\eqref{conv3_ito}
and \eqref{conv4_ito}--\eqref{conv6_ito} we conclude.
\end{proof}


\section{Existence of martingale solutions}
\label{sec:proof1}
This section is devoted to the proof of Theorem~\ref{thm1}. The
proof is based on a Yosida-type approximation on the operators $A$
and $B$. The passage to the limit hinges then on a
lower-semicontinuity arguments, which in turn makes uses of the It\^o
formula for $\haz B$ from Proposition~\ref{prop:ito}. For the sake of
clarity, we subdivide the argument in subsequent steps in the coming
subsections.

\subsection{Approximation}
For every $\lambda>0$ let $A_\lambda,B_\lambda:H\to H$
be the Yosida approximations of $A$ and $B$, respectively, which 
which we recall to be a maximal monotone $\frac1\lambda$-Lipschitz-continuuos 
operators on $H$. The respective resolvents are denoted by 
$J_\lambda^A,J_\lambda^B:H\to H$.
Let us also recall that 
$R_\lambda\in\cL(H,V)$ is a regularizing operator 
converging to the identity in $\cL_s(V,V)$.

The approximated reads as follows:
\begin{align}
  \nonumber
  &\text{find $u_{\lambda}\in\,\mathcal I^{2,2}(H, H)$ such that}\\
  \label{app_prob}
  &\begin{cases}
     \lambda \partial_t u^d_{\lambda} + A_{\lambda}(\partial_t u^d_{\lambda}) 
     + B_\lambda(u_{\lambda}) = F(\cdot, R_\lambda J_\lambda^B(u_{\lambda}))\,,\\
    u_{\lambda}^d(0)=u_{0}\,,\\
    u^s_{\lambda}=G(\cdot,J_\lambda^B(u_{\lambda}))\,.
  \end{cases}
\end{align}
More precisely, 
this means that we look for an
$H$-valued continuous process $u_{\lambda}$ such that
\beq\label{app_prob1}
  u_{\lambda}=u_{0} 
  +\int_0^\cdot\partial_t u_{\lambda}^d(s)\,\d s 
  + \int_0^\cdot G(s, J_\lambda^B(u_{ \lambda}(s)))\,\d W(s)\,,
\eeq
where $u_{\lambda}^d \in L^2_{\cP}(\Omega; H^1(0,T; H))$ is such that 
\beq\label{app_prob2}
  \lambda \partial_t u^d_{\lambda} +
  A_{\lambda}(\partial_t u^d_{\lambda}) + 
  B_\lambda(u_{\lambda}) = F(\cdot, R_\lambda J_\lambda^B(u_{\lambda}))
  \qquad\text{in } L^2_{\cP}(\Omega; L^2(0,T; H))\,.
\eeq

By definition of $A_\lambda$, it is readily seen that
$\lambda I + A_{\lambda}:H\to H$ is Lipschitz-continuous and strongly monotone:
hence, the inverse operator $(\lambda I +A_\lambda)^{-1}:H\to H$ is well-defined and
Lipschitz-continuous as well.
It follows that we can equivalently rewrite the differential relation above as
\[
  \partial_t u^d_{\lambda} = 
  (\lambda I + A_{\lambda})^{-1}(F(\cdot, R_\lambda 
  J_\lambda^B(u_{\lambda})) - B_\lambda(u_\lambda))\,.
\] 
Then, the approximated problem \eqref{app_prob1}--\eqref{app_prob2} can be written in
so-called normal form as
\[
  \d u_{\lambda} 
  = (\lambda I +A_{\lambda})^{-1}(F(\cdot, R_\lambda 
  J_\lambda^B(u_{\lambda})) - B_\lambda(u_\lambda))\,\d t
  + G(\cdot,J_\lambda^B(u_{\lambda}))\,\d W\,, \qquad
  u_{\lambda}(0)=u_{0}\,.
\]
Now, it is clear that by assumption {\bf G} and the fact that $J_\lambda^B:H\to H$
is Lipschitz-continuous, the operator $G(\cdot,J_\lambda^B):[0,T]\times H\to\cL^2(U,H)$ is
Lipschitz-continuous and linearly bounded in $H$, uniformly on $[0,T]$.
Moreover, recall also that 
$(\lambda I +A_{\lambda})^{-1}$ and $B_\lambda$ are Lipschitz-continuous on $H$, 
and that $R_\lambda:H\to V$ is linear continuous
(so in particular Lipschitz-continuous). Hence, bearing in mind that 
by Lemma~\ref{lem:res} we have 
\[
  \norm{J_\lambda^B(x_1)-J_\lambda^B(x_2)}_V^p\leq \frac1{\lambda c_B}\norm{x_1-x_2}_H^2
  \qquad\forall\,x_1,x_2\in H\,,
\]
we deduce that the operator 
\[
  S_\lambda(\cdot,x):=(\lambda I +A_{\lambda})^{-1}(F(\cdot, R_\lambda 
  J_\lambda^B(x)) - B_\lambda(x))\,, \qquad x\in H\,,
\]
is locally Lipschitz-continuous and linearly bounded on $H$.
Indeed, by assumption {\bf F} we have
\begin{align*}
  \norm{S_\lambda(\cdot, x_1)-S_\lambda(\cdot, x_2)}_H&\leq 
  M_\lambda\left(1 + \norm{J_\lambda^B(x_1)}_V^{\frac{p-2}2}
  +\norm{J_\lambda(x_2)}_V^{\frac{p-2}2}\right)\norm{x_1-x_2}_H\\
  &\leq 
  M'_\lambda\left(1 + \norm{x_1}_H^{\frac{p-2}p}
  +\norm{x_2}_H^{\frac{p-2}p}\right)\norm{x_1-x_2}_H \qquad\forall\,x_1,x_2\in H
\end{align*}
and
\[
  \norm{S_\lambda(\cdot, x)}_H\leq M_\lambda'\left(1+  |h_F(\cdot)|  +\norm{x}_H\right)
  \qquad\forall\,x\in H\,,
\]
for certain constants $M_\lambda, M_\lambda'>0$.
Since $\frac{p-2}{p}\leq1$, we can apply the 
classical existence-uniqueness results for SPDEs 
with locally Lipschitz coefficients (see \cite{LiuRo}), and infer that the
approximated problem \eqref{app_prob1}--\eqref{app_prob2}
admits a unique global solution 
\[
  u_\lambda \in \mathcal I^{2,2}(H,H)\,.
\]

\subsection{Uniform estimates}
We prove here some estimates on the approximated solutions
uniformly in $\lambda$. 
To this end, from \eqref{app_prob1}, we can write It\^o's 
formula for $\widehat B_\lambda$ by using Lemma~\ref{lem:ito_lam}, getting 
\begin{align*}
  \widehat B_\lambda(u_{\lambda}(t)) &= 
  \widehat B_\lambda(u_0)
  +\int_0^t
  \left(\partial_tu_{\lambda}^d(s), B_\lambda(u_{\lambda}(s)\right)_H\,\d s\\
  &+\frac12\int_0^t
  \operatorname{Tr}\left[
  G(s,J_\lambda^B(u_{\lambda}(s)))^*
  D_{\mathcal G}B_\lambda(u_{\lambda}(s))G(s,J_\lambda^B(u_{\lambda}(s)))
  \right]\,\d s\\
  &+\int_0^t
  \left(B_\lambda(u_{\lambda}(s)),G(s,J_\lambda^B(u_{\lambda}(s)))\,\d W(s)\right)_H
  \qquad\forall\,t\in[0,T]\,,\quad\P\text{-a.s.}
\end{align*}
Taking equation \eqref{app_prob2} into account, this leads to 
\begin{align}
  \nonumber
  &\widehat B_\lambda(u_{\lambda}(t)) 
  +\lambda\int_0^t\norm{\partial_tu_\lambda^d(s)}_H^2\,\d s
  +\int_0^t\left(A_\lambda(\partial_t u_\lambda^d(s)), \partial_t u_\lambda^d(s)\right)_H\,\d s\\
  &\nonumber= 
  \widehat B_\lambda(u_0)
  +\int_0^t
  \left(F(s,R_\lambda J_\lambda^B(u_\lambda(s))), \partial_tu_{\lambda}^d(s)\right)_H\,\d s
  +\int_0^t
  \left(B_\lambda(u_{\lambda}(s)),G(s,J_\lambda^B(u_{\lambda}(s)))\,\d W(s)\right)_H\\
  &\label{ito_lambda}\quad+\frac12\int_0^t
  \operatorname{Tr}\left[
  G(s,J_\lambda^B(u_{\lambda}(s)))^*
  D_{\mathcal G}B_\lambda(u_{\lambda}(s))G(s,J_\lambda^B(u_{\lambda}(s)))
  \right]\,\d s
  \qquad\forall\,t\in[0,T]\,,\quad\P\text{-a.s.}
\end{align}
On the left-hand side, the coercivity of $A$ in assumption {\bf A} and the definition 
of $\widehat B_\lambda$ and $A_\lambda$ give
\begin{align*}
  &\widehat B_\lambda(u_{\lambda}(t)) +
  \int_0^t\left(A_\lambda(\partial_t u_\lambda^d(s)), \partial_t u_\lambda^d(s)\right)_H\,\d s\\
  &\geq\widehat B(J_\lambda^B(u_{\lambda}(t))) + 
  \int_0^t\left(A_\lambda(\partial_t u_\lambda^d(s)), 
  J_\lambda^A(\partial_t u_\lambda^d(s))\right)_H\,\d s
  +\lambda\int_0^t\norm{A_\lambda(\partial_t u_\lambda^d(s))}_H^2\,\d s\\
  &\geq \widehat B(J_\lambda^B(u_{\lambda}(t))) + 
  c_A\int_0^t\norm{J_\lambda^A(\partial_t u_\lambda^d(s))}_H^2\,\d s
  -c_A^{-1}T
  +\lambda\int_0^t\norm{A_\lambda(\partial_t u_\lambda^d(s))}_H^2\,\d s\,.
\end{align*}
On the right-hand side, first of all it is clear that 
\[
  \widehat B_\lambda(u_0) \leq \widehat B(u_0) \in L^1(\Omega)\,.
\]
Secondly, assumption {\bf F}, 
the weighted Young inequality, and the uniform boundedness 
of $(R_\lambda)_\lambda$ in $\cL(V,V)$
imply that,
for every $\delta>0$,
\begin{align*}
  &\int_0^t
  \left(F(s,R_\lambda J_\lambda^B(u_\lambda(s))), \partial_tu_{\lambda}^d(s)\right)_H\,\d s
  \leq
  \int_0^t
  \norm{F(s,R_\lambda J_\lambda^B(u_\lambda(s)))}_H
  \norm{\partial_tu_{\lambda}^d(s)}_H\,\d s \\
  &\qquad\leq\delta\int_0^t\norm{\partial_t u_\lambda^d(s)}_H^2\,\d s +
  \frac{1}{4\delta}\left(\norm{h_F}_{L^1(0,T)} + 
  C_F\int_0^t\norm{R_\lambda J_\lambda^B(u_\lambda(s))}_V^p\,\d s\right)\\
  &\qquad\leq \delta\int_0^t\norm{\partial_t u_\lambda^d(s)}_H^2\,\d s
  +M_\delta\left(1+\int_0^t\norm{J_\lambda^B(u_\lambda(s))}_V^p\,\d s\right)\,,
\end{align*}
where $M_\delta >0$ is a positive constant independent of $\lambda$.
Noting further that 
\[
  \int_0^t\norm{\partial_t u_\lambda^d(s)}_H^2\,\d s\leq
  2\int_0^t\norm{J_\lambda^A(\partial_t u_\lambda^d(s))}_H^2\,\d s
  +2\lambda^2\int_0^t\norm{A_\lambda(\partial_t u_\lambda^d(s))}_H^2\,\d s\,,
\]
taking $\lambda\in(0,1)$ and choosing $\delta$ sufficient small, 
rearranging the terms in \eqref{ito_lambda} we infer that there
exists a constant $M>0$ independent of $\lambda$ such that
\begin{align*}
  &\widehat B(J_\lambda^B(u_{\lambda}(t))) 
  +\frac{c_A}2\int_0^t\norm{J_\lambda^A(\partial_t u_\lambda^d(s))}_H^2\,\d s
  +\frac\lambda2\int_0^t\norm{\partial_tu_\lambda^d(s)}_H^2\,\d s\\
  &\leq M\left(1+\int_0^t\norm{J_\lambda^B(u_\lambda(s))}_V^p\,\d s\right)
  +\int_0^t
  \left(B_\lambda(u_{\lambda}(s)),G(s,J_\lambda^B(u_{\lambda}(s)))\,\d W(s)\right)_H\\
  &\quad+\frac12\int_0^t
  \operatorname{Tr}\left[
  G(s,J_\lambda^B(u_{\lambda}(s)))^*
  D_{\mathcal G}B_\lambda(u_{\lambda}(s))G(s,J_\lambda^B(u_{\lambda}(s)))
  \right]\,\d s
  \qquad\forall\,t\in[0,T]\,,\quad\P\text{-a.s.}
\end{align*}
At this point, the trace term can be handled using assumption {\bf G}, Lemma~\ref{lem:res},
and the fact that $\norm{D_{\mathcal G}J_\lambda^B(u_\lambda)}_{\cL(H,H)}\leq 1$ as
\begin{align*}
  &\int_0^t\operatorname{Tr}\left[
  G(s,J_\lambda^B(u_{\lambda}(s)))^*
  D_{\mathcal G}B_\lambda(u_{\lambda}(s))G(s,J_\lambda^B(u_{\lambda}(s)))
  \right]\,\d s\\
  &=\int_0^t\operatorname{Tr}\left[G(s,J_\lambda^B(u_{\lambda}(s)))
  G(s,J_\lambda^B(u_{\lambda}(s)))^*
  D_{\mathcal G}B(J^B_\lambda(u_{\lambda}(s)))
  D_{\mathcal G}J^B_\lambda(u_{\lambda}(s))
  \right]\,\d s\\
  &=\int_0^t\operatorname{Tr}\left[L(s,J_\lambda^B(u_{\lambda}(s)))
  D_{\mathcal G}J^B_\lambda(u_{\lambda}(s))
  \right]\,\d s\\
  &\leq\int_0^t\norm{L(s,J_\lambda^B(u_{\lambda}(s)))}_{\cL^1(H,H)}
  \norm{D_{\mathcal G}J^B_\lambda(u_{\lambda}(s))}_{\cL(H,H)}\,\d s\\
  &\leq\norm{h_G}_{L^1(0,T)} + C_G\int_0^t\norm{J^B_\lambda(u_{\lambda}(s))}_V^p\,\d s
\end{align*}
Furthermore, let $x\in V$ be arbitrary and note that by {\bf B1} we have 
\begin{align*}
  \widehat B(x)&\geq 
  \widehat B_\eps(x) = \int_0^1\left(B_\eps(rx),x\right)_H\,\d r=
  \int_0^1r^{-1}\left(B_\eps(rx),J_\eps^B(rx)\right)_H\,\d r
  +\eps\int_0^1r^{-1}\norm{B_\eps(rx)}_H^2\,\d r\\
  &\geq c_B\int_0^1r^{-1}\norm{J_\eps^B(rx)}^p_V\,\d r \qquad\forall\,\eps>0\,.
\end{align*}
Letting $\eps\to0$, by Lemma~\ref{lem:res} we have 
$J_\eps^B(rx)\wto rx$ in $V$ for every $r\in[0,1]$, so that 
by the Fatou Lemma it follows that 
\beq\label{coerc_B}
  \widehat B(x) \geq \frac{c_B}p\norm{x}_V^p \qquad\forall\,x\in V\,.
\eeq
Consequently, possibly updating the value of the constant $M$, we are left with
\begin{align}
  \nonumber
  \widehat B(J_\lambda^B(u_{\lambda}(t))) 
  +\int_0^t\norm{J_\lambda^A(\partial_t u_\lambda^d(s))}_H^2\,\d s
  &\leq M\left(1+\int_0^t\widehat B(J_\lambda^B(u_{\lambda}(s))) \,\d s\right)\\
  \label{est_lam_aux}
  &+\int_0^t
  \left(B_\lambda(u_{\lambda}(s)),G(s,J_\lambda^B(u_{\lambda}(s)))\,\d W(s)\right)_H
\end{align}
for every $t\in[0,T]$, $\P$-almost surely.
The stochastic integral is a martingale, so that 
taking expectations and using the Gronwall Lemma yield
\beq
  \label{est1}
  \norm{\widehat B(J_\lambda^B(u_\lambda))}_{L^\infty_\cP(0,T; L^1(\Omega))}
  +\norm{J_\lambda^A(\partial_t u_\lambda^d)}_{L^2_\cP(\Omega; L^2(0,T; H))} \leq M\,.
\eeq
By \eqref{coerc_B} and the linear growth of $A$ in assumption {\bf A}, this implies in turn that
\beq
  \label{est2}
  \norm{J_\lambda^B(u_\lambda)}_{L^\infty_\cP(0,T; L^p(\Omega; V))}
  +\norm{A_\lambda(\partial_t u_\lambda^d)}_{L^2_\cP(\Omega; L^2(0,T; H))} \leq M\,.
\eeq
Consequently, assumptions {\bf F--G} yield directly 
\beq\label{est3}
  \norm{F(\cdot, R_\lambda J_\lambda^B(u_\lambda))}_{L^2_\cP(\Omega; L^2(0,T; H))}
  +\norm{G(\cdot, J_\lambda^B(u_\lambda))}_{L^\infty_\cP(0,T; L^p(\Omega; \cL^2(U,H)))} \leq M\,,
\eeq
from which we deduce, by comparison in \eqref{app_prob2}, that 
\beq
\label{est4}
  \norm{B_\lambda(u_\lambda)}_{L^2_\cP(\Omega; L^2(0,T; H))}\leq M\,.
\eeq
At this point, going back to \eqref{est_lam_aux} and using the estimates
\eqref{est3}--\eqref{est4} together with the Burkholder-Davis-Gundy inequality 
on the stochastic integral, we get by a standard argument that 
\beq
  \label{est5}
   \norm{\widehat B(J_\lambda^B(u_\lambda))}_{L^1_\cP(\Omega; L^\infty(0,T))}
   +\norm{J_\lambda^B(u_\lambda)}_{L^p_\cP(\Omega; L^\infty(0,T; V))} \leq M\,,
\eeq
which in turn implies, again by assumptions {\bf F--G}, that, setting $p_\nu:=p/\nu$,
\beq\label{est6}
  \norm{F(\cdot, R_\lambda J_\lambda^B(u_\lambda))}_{L^2_\cP(\Omega; L^2(0,T; H))}
  +\norm{G(\cdot, J_\lambda^B(u_\lambda))}_{L^{p_\nu}_\cP(\Omega; 
  L^\infty(0,T; \cL^2(U,H)))} \leq M\,.
\eeq
It follows in particular by the H\"older inequality that, setting $q_\nu:={2p_\nu}/{(p_\nu+2)}$,
\[
  \norm{B_\lambda(u_\lambda)G(\cdot, J_\lambda^B(u_\lambda))}_{
  L^{q_\nu}_\cP(\Omega; L^2(0,T; \cL^2(U,\erre)))}\leq M\,,
\]
where by assumption on $\nu$ we always have that $q_\nu>1$.
Hence, going back again to \eqref{est_lam_aux}
the Burkholder-Davis-Gundy inequality and the Gronwall lemma allow
to refine the moment estimates as
\beq
  \label{est_refined1}
  \norm{\widehat B(J_\lambda^B(u_\lambda))}_{L^{q_\nu}_\cP(\Omega; L^\infty(0,T))}
   +\norm{J_\lambda^B(u_\lambda)}_{L^{pq_\nu}_\cP(\Omega; L^\infty(0,T; V))} \leq M
\eeq
and
\beq
  \label{est_refined2}
  \norm{J_\lambda^A(\partial_t u_\lambda^d)}_{L^{2q_\nu}_\cP(\Omega; L^2(0,T; H))}
  +\norm{A_\lambda(\partial_t u_\lambda^d)}_{L^{2q_\nu}_\cP(\Omega; L^2(0,T; H))} \leq M\,,
\eeq
Lastly, the classical result \cite[Lem.~2.1]{fland-gat} by {\sc Flandoli \& Gatarek} ensures that 
\beq
  \label{est7}
  \norm{I_\lambda:=
  \int_0^\cdot G(s, J_\lambda^B(u_\lambda(s)))}_{L^{p_\nu}_\cP(\Omega; 
  W^{\eta,p_\nu}(0,T; H))}
  \leq M_\eta \qquad\forall\,\eta\in(0,1/2)\,,
\eeq
yielding by comparison in \eqref{app_prob1} that 
\beq\label{est8}
  \norm{u_\lambda}_{L^2_\cP(\Omega; H^1(0,T; H))+
  L^{p_\nu }_\cP(\Omega; W^{\eta,p_\nu}(0,T; H))}
  \leq M_\eta \qquad\forall\,\eta\in(0,1/2)\,.
\eeq

\subsection{Passage to the limit}
\label{ssec:pass_lim}
First of all, note that by assumption on $\nu$ in {\bf G}, 
we always have that $p_\nu=p/\nu > 2$ and $q_\nu>1$: hence, 
we can fix $\eta \in (\frac1{p_\nu}, \frac12)$, so that $\eta p_\nu >1$.
Since $V\embed H$ compactly,
by the classical Aubin-Lions-Simon compactness results \cite[Cor.~4--5, p.~85]{Simon} we have
\begin{align*}
  W^{\eta,p_\nu}(0,T; H)\embed C^0([0,T]; V^*)
  \qquad&\text{compactly}\,,\\
  L^\infty(0,T; V) \cap \left( H^1(0,T; H) +
  W^{\eta,p_\nu}(0,T; H)\right)\embed C^0([0,T]; H)
  \qquad&\text{compactly}\,.
\end{align*}
Let us put then
\[
  \mathcal X:= C^0([0,T]; H) \times C^0([0,T]; V^*) \times C^0([0,T]; U)\,.
\]
By the estimates \eqref{est5}--\eqref{est8} and the compactness inclusions above, 
using the Prokhorov theorem we readily infer that 
\[
  \text{the laws of}\quad\{(J_\lambda^B(u_\lambda), I_\lambda, W)\}_\lambda \quad
  \text{are tight on $\mathcal X$}\,.
\]
By the Skorokhod theorem \cite[Thm.~2.7]{ike-wata}, there exist
then a probability space $(\hat\Omega, \hat\cF,\hat\P)$, a sequence
of measurable random variables
\[
  \phi_\lambda:(\hat\Omega,\hat\cF)\to(\Omega,\cF)\,,\qquad\lambda>0\,,
\]
with $\P\circ\phi_\lambda=\hat\P$ for all $\lambda>0$, and 
some measurable random variables 
\[
  (\hat u, \hat I, \hat W):(\hat\Omega, \hat\cF) \to \mathcal X
\]
such that, setting $\hat u_\lambda:=u_\lambda\circ \phi_\lambda$,
\begin{align}
  \label{conv1}
  J_\lambda^B(\hat u_\lambda) \to \hat u 
  \qquad&\text{in } C^0([0,T]; H)\,, \quad\hat\P\text{-a.s.}\,,\\
  \label{conv2}
  \hat I_\lambda:=I_\lambda\circ\phi_\lambda \to \hat I
  \qquad&\text{in } C^0([0,T]; V^*)\,, \quad\hat\P\text{-a.s.}\,,\\
  \label{conv3}
  \hat W_\lambda:=W\circ\phi_\lambda \to \hat W
   \qquad&\text{in } C^0([0,T]; U)\,, \quad\hat\P\text{-a.s.}\,.
\end{align}
Furthermore,
by \eqref{est1}--\eqref{est8} and the fact that 
composition with $\phi_\lambda$ preserves the laws, we also infer the convergences
\begin{align}
  \label{conv4}
  J_\lambda^B(\hat u_\lambda) \to \hat u
  \qquad&\text{in } L^\ell(\hat\Omega; C^0([0,T]; H))\quad
  \forall\,\ell\in[1,pq_\nu)\,,\\
  \label{conv5}
  J_\lambda^B(\hat u_\lambda) \wstarto \hat u
  \qquad&\text{in } L^{pq_\nu}(\hat\Omega; L^\infty(0,T; V))\,,\\
  \label{conv6}
  \partial_t\hat u_\lambda^d \wto \hat u'
  \qquad&\text{in } L^2(\hat\Omega; L^2(0,T; H))\,,\\
  \label{conv7}
  B_\lambda(\hat u_\lambda) \wto \hat w
  \qquad&\text{in } L^2(\hat\Omega; L^2(0,T; H))\,,\\
  \label{conv8}
  A_\lambda(J_\lambda^A(\hat u_\lambda)) \wto \hat v
  \qquad&\text{in } L^2(\hat\Omega; L^2(0,T; H))\,,
\end{align}
for some $\hat u',\hat w, \hat v\in L^2(\hat\Omega; L^2(0,T; H))$.
Now,
 noting that $pq_\nu>2$, 
by the 
strong-weak closure of the maximal monotone operator $B$  we readily get that
\[
  \hat u\in D(B)\,, \quad
  \hat w= B(\hat u) \qquad\text{a.e.~in } \hat\Omega\times(0,T)\,.
\]
Consequently, by the strong monotonicity of $B$ in {\bf B1}
we obtain
\[
  c_B\norm{J_\lambda^B(\hat u_\lambda)-\hat u}_{L^p(\hat \Omega; L^p(0,T; V))}^p
  \leq\E\int_0^T
  \left(B_\lambda(\hat u_\lambda(s)) - \hat w(s), J_\lambda^B(\hat u_\lambda(s))-\hat u(s)
  \right)_H\,\d s \to 0\,,
\]
so that 
\beq
  \label{conv9}
  J_\lambda^B(\hat u_\lambda)\to \hat u \qquad\text{in } L^p(\hat \Omega; L^p(0,T; V))\,.
\eeq
Since $R_\lambda\to I$ in $\cL_s(V,V)$, the family 
$(R_\lambda)_\lambda$ is uniformly bounded in $\cL(V,V)$
by the Banach-Steinhaus theorem, hence the strong convergence \eqref{conv9} yields
\[
  R_\lambda J_\lambda^B(\hat u_\lambda)\to \hat u 
  \qquad\text{in } L^p(\hat \Omega; L^p(0,T; V))\,.
\]
At this point, using {\bf F} it is immediate to infer that 
\beq
  \label{conv10}
  F(\cdot, R_\lambda J_\lambda^B(\hat u_\lambda))\to F(\cdot, \hat u) 
  \qquad\text{in } L^2(\hat \Omega; L^2(0,T; H))\,,
\eeq
while {\bf G} gives
\begin{align}
  \label{conv11}
  G(\cdot, J_\lambda^B(\hat u_\lambda))\to G(\cdot, \hat u) 
  \qquad&\text{in } L^\ell(\hat \Omega; C^0([0,T]; \cL^2(U,H))) \quad\forall\,\ell\in[1,pq_\nu)\,,\\
  \label{conv11'}
  G(\cdot, J_\lambda^B(\hat u_\lambda))\to G(\cdot, \hat u) 
  \qquad&\text{in } L^p(\hat \Omega; L^p(0,T; \cL^2(U,H)))\,.
\end{align}

Now, by definition of $\phi_\lambda$, from \eqref{app_prob1}--\eqref{app_prob2} we have
\begin{align}
  \label{app_prob1_hat}
  &\hat u_\lambda(t)=u_{0} + 
  \int_0^t\partial_t \hat u_\lambda^d(s)\,\d s
  +\hat I_\lambda(t) \quad\text{in } H &&\forall\,t\in[0,T]\,,\quad\hat\P\text{-a.s.}\,,\\
  \label{app_prob2_hat}
  &\lambda\partial_t \hat u_\lambda^d + A_\lambda(\partial_t \hat u_\lambda^d)
  +B_\lambda(\hat u_\lambda) = F(\cdot, R_\lambda J_\lambda^B(\hat u_\lambda))
  \quad\text{in } H
  &&\text{a.e.~in } \hat\Omega\times(0,T)\,.
\end{align}
Let us introduce on the probability space $(\hat\Omega, \hat\cF, \hat \P)$ the filtration 
\[
  \hat\cF_{\lambda,t}:=
  \sigma\left\{\hat u_\lambda(s), \hat I_\lambda(s), \hat W_\lambda(s):\;s\leq t\right\}\,,
  \qquad t\in[0,T]\,.
\]
Using a classical argument 
(see \cite[\S~8]{dapratozab} and \cite{vall-zimm}), we have
that $\hat W_\lambda$ is an $(\hat\cF_{\lambda,t})_t$-cylindrical Wiener process.
Moreover, following the approach in \cite[\S~4.5]{SWZ18} and \cite{ScarStef-SDNL},
using the fact that $\phi_\lambda$ preserves the laws and
comparing \eqref{app_prob1} and \eqref{app_prob2}, we also deduce that,
possibly enlarging the filtration $(\hat\cF_{\lambda,t})_t$,
$\hat I_\lambda$ is the square-integrable martingale 
\[
  \hat I_\lambda = \int_0^\cdot G(s,J_\lambda^B(\hat u_\lambda(s)))\,\d \hat W(s)\,.
\]

At this point, letting $\lambda\searrow0$ in \eqref{app_prob1_hat}--\eqref{app_prob2_hat},
thanks to the convergences \eqref{conv1}--\eqref{conv11} we get 
\begin{align}
  \label{prob1_hat}
  &\hat u(t)=\hat u_{0} + 
  \int_0^t\hat u'(s)\,\d s
  +\hat I(t) \quad\text{in } V^* &&\forall\,t\in[0,T]\,,\quad\hat\P\text{-a.s.}\,,\\
  \label{prob2_hat}
  &\hat v
  +B(\hat u) = F(\cdot, \hat u) \quad\text{in } H
  &&\text{a.e.~in } \hat\Omega\times(0,T)\,.
\end{align}
We introduce the limiting filtration 
\[
  \hat\cF_{t}:=
  \sigma\left\{\hat u(s), \hat I(s), \hat W(s):\;s\leq t\right\}\,,
  \qquad t\in[0,T]\,,
\]
as well as the corresponding progressive $\sigma$ algebra $\hat \cP$. 
The strong convergence \eqref{conv3} and 
the representation results \cite[\S~8]{dapratozab} ensure again
that $\hat W$ is a $(\hat\cF_t)_t$-cylindrical Wiener process.
Furthermore, proceeding as in \cite[\S~4.5]{SWZ18} and \cite{ScarStef-SDNL},
possibly enlarging $(\hat\cF_{t})_t$ we have the representation 
\[
  \hat I = \int_0^\cdot G(s,\hat u(s))\,\d \hat W(s)\,.
\]
In particular, this shows a posteriori that $\hat I \in L^p_{\hat \cP}(\hat\Omega; C^0([0,T]; H))$.
Consequently, equation \eqref{prob1_hat} yields 
\beq
\label{1_lim}
  \hat u(t) = u_0 + \int_0^t\hat u'(s)\,\d s + \int_0^tG(s,\hat u(s))\,\d \hat W(s)
  \qquad\forall\,t\in[0,T]\,,\quad\hat\P\text{-a.s.}
\eeq
We deduce that
\[
  \hat u\in \mathcal I^{2,p}(H,H)\,, \qquad \partial_t \hat u^d = \hat u'\,, \qquad
  \hat u^s=G(\cdot, \hat u)\,.
\]

\subsection{Identification of the nonlinearity $A$}
\label{ssec:A}
The last thing that we have to show is that $\hat v \in A(\partial_t \hat u^d)$
almost everywhere. To this end, from the It\^o formula \eqref{ito_lambda},
taking expectations and fixing $t=T$
we immediately deduce 
\begin{align*}
  &\hat\E\widehat B_\lambda(\hat u_{\lambda}(T)) 
  +\lambda\hat\E\int_0^T\norm{\partial_t\hat u_\lambda^d(s)}_H^2\,\d s
  +\hat\E\int_0^T\left(A_\lambda(\partial_t \hat u_\lambda^d(s)), 
  \partial_t \hat u_\lambda^d(s)\right)_H\,\d s\\
  &= 
  \widehat B_\lambda(u_0)
  +\hat\E\int_0^T
  \left(F(s,R_\lambda J_\lambda^B(\hat u_\lambda(s))), 
  \partial_t\hat u_{\lambda}^d(s)\right)_H\,\d s\\
  &\qquad+\frac12\hat\E\int_0^T
  \operatorname{Tr}\left[
  G(s,J_\lambda^B(\hat u_{\lambda}(s)))^*
  D_{\mathcal G}B_\lambda(\hat u_{\lambda}(s))G(s,J_\lambda^B(\hat u_{\lambda}(s)))
  \right]\,\d s\,.
\end{align*}
Now, by lower semicontinuity we have
\[
  \liminf_{\lambda\searrow0}\left(
  \hat\E\widehat B_\lambda(\hat u_{\lambda}(T)) 
  +\lambda\hat\E\int_0^T\norm{\partial_t\hat u_\lambda^d(s)}_H^2\,\d s
  \right)
  \geq 
  \hat\E\widehat B(\hat u(T))\,,
\]
while from the convergences \eqref{conv6} and \eqref{conv10} it follows 
\[
  \lim_{\lambda\searrow0}
  \hat\E\int_0^T
  \left(F(s,R_\lambda J_\lambda^B(\hat u_\lambda(s))), 
  \partial_t\hat u_{\lambda}^d(s)\right)_H\,\d s
  = 
 \hat\E\int_0^T
  \left(F(s,\hat u(s)), 
  \partial_t\hat u^d(s)\right)_H\,\d s\,.
\]
As the initial term, we have $\widehat B_\lambda(u_0)\to \widehat B(u_0)$.
Furthermore, for the trace term we note that
\begin{align*}
  \operatorname{Tr}
  \left[G^*(\cdot,J_\lambda^B(\hat u_\lambda))
  D_{\mathcal G}B_\lambda(\hat u_\lambda)G(\cdot,J_\lambda^B(\hat u_\lambda))\right]
  &=\operatorname{Tr}
  \left[G(\cdot,J_\lambda^B(\hat u_\lambda))G^*(\cdot,J_\lambda^B(\hat u_\lambda))
  D_{\mathcal G}B_\lambda(\hat u_\lambda)\right]\\
  &=\operatorname{Tr}
  \left[L(\cdot,J_\lambda^B(\hat u_\lambda))
  D_{\mathcal G}J^B_\lambda(\hat u_\lambda)\right]\,.
\end{align*}
By the strong convergence \eqref{conv9}, the continuity of $L$ in
assumption {\bf G}, and Lemma~\ref{lem:res}, we have
\begin{align*}
  L(\cdot,J_\lambda^B(\hat u_\lambda)) \to L(\cdot,\hat u) \qquad&\text{in } \cL^1(H,H)\,,
  \quad\text{a.e.~in } \hat \Omega\times(0,T)\,,\\
  D_{\mathcal G}J^B_\lambda(\hat u_\lambda) \to I \qquad&\text{in } \cL_s(H,H)\,,
  \quad\text{a.e.~in } \hat \Omega\times(0,T)\,,
\end{align*}
so that 
\[
  \operatorname{Tr}
  \left[L(\cdot,J_\lambda^B(\hat u_\lambda))
  D_{\mathcal G}J^B_\lambda(\hat u_\lambda)\right] \to 
  \operatorname{Tr}
  \left[L(\cdot,\hat u)\right] \qquad\text{a.e.~in } \hat \Omega\times(0,T)\,.
\]
Noting also that by assumption {\bf G} we have
\begin{align*}
  |\operatorname{Tr}
  \left[L(\cdot,J_\lambda^B(\hat u_\lambda))
  D_{\mathcal G}J^B_\lambda(\hat u_\lambda)\right]|
  &\leq\norm{L(\cdot,J_\lambda^B(\hat u_\lambda))}_{\cL^1(H,H)}
  \norm{D_{\mathcal G}J^B_\lambda(\hat u_\lambda)}_{\cL(H,H)}\\
  &\leq
  \norm{L(\cdot,J_\lambda^B(\hat u_\lambda))}_{\cL^1(H,H)}\\
  &\leq h_G(\cdot) + C_G\norm{J_\lambda^B(\hat u_\lambda)}_V^p\,,
\end{align*}
convergence \eqref{conv9} and the
Vitali convergence theorem yield
\[
 \operatorname{Tr}
  \left[G^*(\cdot,J_\lambda^B(\hat u_\lambda))
  D_{\mathcal G}B_\lambda(\hat u_\lambda)G(\cdot,J_\lambda^B(\hat u_\lambda))\right]  \to 
  \operatorname{Tr}
  \left[L(\cdot,\hat u)\right] \qquad\text{in } L^1_{\hat\cP}(\hat\Omega; L^1(0,T))\,.
\]
Putting all this information together, we are left with 
\begin{align}
  \nonumber
  &\limsup_{\lambda\searrow0}
  \hat\E\int_0^T\left(A_\lambda(\partial_t \hat u_\lambda^d(s)), 
  \partial_t \hat u_\lambda^d(s)\right)_H\,\d s\\
  &\leq 
  \label{limsup_lam}
  \widehat B(u_0) - \hat\E\widehat B(\hat u(T))
  + \hat\E\int_0^T
  \left(F(s,\hat u(s)), 
  \partial_t\hat u^d(s)\right)_H\,\d s
  +\frac12\hat\E\int_0^T
  \operatorname{Tr}\left[L(\cdot,\hat u(s))\right]\,\d s\,.
\end{align}
At this point, from equation \eqref{1_lim} and 
the generalized It\^o formula in Proposition~\ref{prop:ito} we infer that 
\begin{align}
  \nonumber
  &\hat\E\widehat B(\hat u(T)) +
  \hat\E\int_0^T\left(\hat v(s), 
  \partial_t \hat u^d(s)\right)_H\,\d s\\
  &\qquad
  \label{limsup_lim}
  = \widehat B(u_0)
  +\hat\E\int_0^T
  \left(F(s,\hat u(s)), 
  \partial_t\hat u^d(s)\right)_H\,\d s
  +\frac12\hat\E\int_0^T
  \operatorname{Tr}\left[L(\cdot,\hat u(s))\right]\,\d s\,.
\end{align}
Hence, comparing \eqref{limsup_lam} and \eqref{limsup_lim}
we infer that 
\[
  \limsup_{\lambda\searrow0}
  \hat\E\int_0^T\left(A_\lambda(\partial_t \hat u_\lambda^d(s)), 
  \partial_t \hat u_\lambda^d(s)\right)_H\,\d s
  \leq 
  \hat\E\int_0^T\left(\hat v(s), 
  \partial_t \hat u^d(s)\right)_H\,\d s\,.
\]
Together with the weak convergences \eqref{conv6}--\eqref{conv8}
and the maximal monotonicity of $A$, 
this ensures indeed that $\hat v \in A(\hat u)$
almost everywhere in $\hat\Omega\times(0,T)$, and concludes the proof
of Theorem~\ref{thm1}.


\section{Uniqueness and existence of probabilistically strong solutions}
\label{sec:proof2}
This last section is devoted to the proof of Theorem~\ref{thm2},
showing that existence and uniqueness of 
probabilistically strong solutions hold
under the additional assumptions that $B:V\to V^*$ and 
$G(t,\cdot):H\to\cL^2(U,H)$
are linear and continuous, and $A:H\to 2^H$ is strongly monotone.
Let us point out that since by {\bf B1} the operator 
$B:V\to V^*$ is a subdifferential, 
we automatically have that $B$ is self-adjoint, and necessarily $p=2$ and
\[
  \widehat B(x)=\frac12\ip{B(x)}{x}_{V^*,V}\,, \qquad
  D_{\mathcal G}B(x)=B\in\cL(V,V^*)\,, \qquad\forall\,x\in V\,.
\]
Moreover, the strong monotonicity of $A$ reads
\[
  \left(y_1-y_2, x_1-x_2\right)_H\geq c_A\norm{x_1-x_2}_H^2
  \qquad\forall\,x_i\in D(A)\,,\quad\forall\,y_i\in A(x_i)\,, \quad i=1,2\,.
\]

First of all, we show uniqueness of martingale solutions.
Let $(\hat u_1,\hat u_1^d,\hat v_1)$
and $(\hat u_2, \hat u_2^d, \hat v_2)$ be two martingale solutions 
in the sense of Theorem~\ref{thm1} on the same stochastic 
basis $(\hat\Omega, \hat\cF, (\hat\cF_t)_t, \hat\P, \hat W)$,
with respect to some initial data $u_{0,1}, u_{0,2}\in V$.
Then, using the fact that $B$ is linear we get 
\begin{align*}
  &\hat u_1-\hat u_2 = u_{0,1}-u_{0,2} + \int_0^\cdot\partial_t(\hat u_1^d-\hat u_2^d)(s)\,\d s
  +\int_0^\cdot G(s,(\hat u_1-\hat u_2)(s))\,\d\hat W(s)\,,\\
   &\hat v_1-\hat v_2 + B(\hat u_1-\hat u_2) = F(\cdot,\hat u_1) - F(\cdot,\hat u_2)\,.
\end{align*}
The generalized It\^o formula in Proposition~\ref{prop:ito} gives then
\begin{align*}
  &\widehat B((\hat u_1-\hat u_2)(t)) + 
  \int_0^t\left((\hat v_1-\hat v_2)(s), \partial_t(\hat u_1^d-\hat u_2^d)(s)\right)_H\,\d s \\
  &=\widehat B(u_{0,1}-u_{0,2}) + 
  \int_0^t\left(F(s,\hat u_1(s))-F(s,\hat u_2(s)), \partial_t(\hat u_1^d-\hat u_2^d)(s)\right)_H\,\d s\\
  &\quad+\frac12\int_0^t\operatorname{Tr}
  \left[L(s,(\hat u_1-\hat u_2)(s))\right]\,\d s
  +\int_0^t\left(B((\hat u_1-\hat u_2)(s)), G(s,(\hat u_1-\hat u_2)(s))\,\d W(s)\right)_H
\end{align*}
for every $t\in[0,T]$, $\hat P$-almost surely.
On the left-hand side,
by the coercivity of $\widehat B$ \eqref{coerc_B} we have
\[
  \widehat B((\hat u_1-\hat u_2)(t)) \geq \frac{c_B}{2}\norm{(\hat u_1-\hat u_2)(t)}_V^2\,,
\]
while the strong monotonicity of $A$ yields
\[
  \int_0^t\left((\hat v_1-\hat v_2)(s), \partial_t(\hat u_1^d-\hat u_2^d)(s)\right)_H\,\d s
  \geq c_A\int_0^t\norm{\partial_t(\hat u_1^d-\hat u_2^d)(s)}_H^2\,\d s\,.
\]
On the right-hand side, the first term gives, by assumption {\bf B1--B2},
\[
  \widehat B(u_{0,1}-u_{0,2}) \leq \frac12\norm{B}_{\cL(V,V^*)}\norm{u_{0,1}-u_{0,2}}_V^2\,,
\]
while the second and third ones
can be handled using the Young inequality and
{\bf F--G} as
\begin{align*}
  &\int_0^t\left(F(s,\hat u_1(s))-F(s,\hat u_2(s)), \partial_t(\hat u_1^d-\hat u_2^d)(s)\right)_H\,\d s
  +\frac12\int_0^t\operatorname{Tr}
  \left[L(s,(\hat u_1-\hat u_2)(s))\right]\,\d s\\
  &\leq \frac{c_A}{2}\int_0^t\norm{\partial_t(\hat u_1^d-\hat u_2^d)(s)}_H^2\,\d s
  +\left(\frac{C_F}{2c_A} + \frac{C_G}{2}\right)
  \int_0^t\norm{(\hat u_1-\hat u_2)(s)}_V^2\,\d s\,.
\end{align*}
Taking expectations we infer then 
\begin{align*}
  &\frac{c_B}{2}\E\norm{(\hat u_1-\hat u_2)(t)}_V^2
  +\frac{c_A}2\E\int_0^t\norm{\partial_t(\hat u_1^d-\hat u_2^d)(s)}_H^2\,\d s\\
  &\leq\frac{\norm{B}_{\cL(V,V^*)}}{2}\norm{u_{0,1}-u_{0,2}}_V^2
  +\left(\frac{C_F}{2c_A} + \frac{C_G}{2}\right)
  \E\int_0^t\norm{(\hat u_1-\hat u_2)(s)}_V^2\,\d s
  \qquad\forall\,t\in[0,T]\,.
\end{align*}
The Gronwall lemma ensures then the required continuous dependence result.

As far as uniqueness is concerned, the continuous dependence property 
directly implies that if $u_{0,1}=u_{0,2}$ then
\[
  \hat u_1=\hat u_2 \quad\text{in } L^\infty(0,T; L^2(\hat\Omega; V))\,, \qquad
  \partial_t u_1^d=\partial_tu_2^d \quad\text{in } L^2_{\hat \cP}(\hat\Omega; L^2(0,T; H))\,.
\]
The second equality yields straightaway that 
\[
  \hat u_1^d = \hat u_2^d \quad\text{in } L^2_{\hat \cP}(\hat\Omega; C^0([0,T]; H))\,,
\]
while the first equality and assumption {\bf G} implies 
\[
  G(\cdot, \hat u_1)=G(\cdot, \hat u_2) \quad\text{in } L^2_{\hat \cP}(\hat\Omega; L^2(0,T; \cL^2(U,H)))\,,
\]
from which also 
\[
  \int_0^\cdot G(s, \hat u_1(s))\,\d\hat W(s)
  = \int_0^\cdot G(s, \hat u_2(s))\,\d\hat W(s) \quad\text{in } L^2_{\hat \cP}(\hat \Omega; C^0([0,T]; H))\,.
\]
Consequently, we deduce that 
\[
  \hat u_1 = \hat u_2 \quad\text{in } L^2_{\hat \cP}(\hat\Omega; C^0([0,T]; H))\,,
\]
and pathwise uniqueness of martingale solutions holds.

In order to conclude the proof of Theorem~\ref{thm2},
we are only left to show existence of probabilistically strong solutions.
This can be done in a classical way, using existence and uniqueness of martingale solutions.
In particular, we recall the following lemma due to
{\sc Gy\"ongy \& Krylov} \cite[Lem.~1.1]{gyo-kry}.
\begin{lem}
  Let $\mathcal X$ be a Polish space and $(Z_n)_n$ be a sequence
  of $\mathcal X$-valued random variables. Then $(Z_n)_n$ converges
  in probability if and only if
  for any pair of subsequences $(Z_{n_k})_k$ and $(Z_{n_j})_j$, there exists 
  a joint sub-subsequence $(Z_{n_{k_\ell}}, Z_{n_{j_\ell}})_\ell$ converging 
  in law to a probability measure $\nu$ on $\mathcal X\times\mathcal X$ such
  that $\nu(\{(z_1,z_2)\in\mathcal X\times\mathcal X: z_1=z_2\})=1$.
\end{lem}
Under the additional assumptions of Theorem~\ref{thm2},
recalling the proof of Theorem~\ref{thm1},
it is readily seen that 
the Skorokhod Theorem and the just proved pathwise uniqueness of martingale solutions 
to the limit problem
ensure exactly the condition of the lemma above.
It follows then that all the arguments 
for passing to the limit as $\lambda\searrow0$
in Subsections~\ref{ssec:pass_lim}--\ref{ssec:A}
can be replicated on the original probability space $(\Omega,\cF,\P)$.
This ensures then existence of probabilistically strong solutions, 
and Theorem~\ref{thm2} is proved.

\section{Application}\label{sec:appl}

We conclude our discussion by presenting a concrete case of an SPDE
to which the abstract theory applies. Let us consider the initial and boundary
value problem
\begin{align}
 &\alpha(\partial_tu^d)\,\d t + u^s \, \d W- {\rm div}\,(D\haz \beta_1(\nabla u))
  \,  \d t +  \haz \beta_0'(u)\, \d t\ni f(\cdot,u,\nabla u) \, \d t +
  G(\cdot,u) \, \d W , \nonumber\\
&\qquad  \text{in} \ \ \Omega  \times (0,T)\times
  \OO\,, 
  \label{eq:00}\\
& u = 0 \quad \text{on} \ \ \Omega  \times (0,T)\times
  \partial \OO\,, 
  \label{eq:001}\\
&  u(0,\cdot) =u_0 \quad \text{on} \ \ \Omega \times
   \OO\,.
  \label{eq:002}
\end{align}
Here, the solution is a scalar function $u : \Omega \times (0,T)\times \OO \to \Rz$ where $\OO\subset \Rz^d$
is an open, bounded, and Lipschitz domain.
The possibly multivalued map
$\alpha : \Rz\to 2^\Rz$ is assumed to be strongly maximal monotone and
linearly bounded
with $0\in \alpha(0)$. The potential $\haz \beta_1 \in C^2(\Rz^d)$ is
asked to be nonnegative, convex, minimized at $0$, and with polynomial
growth. More precisely, we assume to be given $p\geq 2$ and $c_{\beta}>0$ such that 
$$
|D^2\haz \beta_1(\xi)|\leq \frac{1}{c_{\beta}}(1+ |\xi|^{p-2})\,,  
\quad (D\haz \beta_1(\xi) - D\haz \beta_1(\eta))\cdot (\xi -\eta)\geq
c_{\beta}|\xi-\eta|^p\qquad \forall\, \xi, \eta \in \Rz^d\,.
$$ 
The potential $\haz \beta_0\in C^2(\Rz)$ is assumed to be nonnegative,
convex, minimized at $0$, with 
$\haz \beta_0''$ bounded by a polynomial of order $p-2$. 
We assume $f : (0,T)\times \Rz\times \Rz^d\to \Rz$ to be
Carath\'eodory and
uniformly Lipschitz in the third variable with $t \mapsto f(t,0,0)\in L^1(0,T)$, 
$G$ to fulfil 
Assumption {\bf G} with 
$$H=L^2(\OO), \qquad V=W^{1,p}_0(\OO)$$
and the initial datum $u_0$ to belong to $W^{1,p}_0(\OO)$. 
Note that the homogenous
Dirichlet boundary conditions are chosen here for the sake of
definiteness, other choices also being possible.
  Note assumption {\bf G} allows for most of the classical 
  choices of the noise coefficient. These include, 
  but are not limited to, the natural case of 
  additive noise $G\in \cL^2(U,V)$ as well as 
  the case of multiplicative noise of superposition
  type. An explicit form for $G$ could be
$$G(u) \, {\rm d} W = \sum_{j=1}^\infty h_j(u) \, {\rm d} W_j$$
where $W_j$ are independent real Brownian motions and $h_j:H \to H$
are $L_{j}$-Lipschitz continuous functions with $\sum_{j}L_j^2<\infty$
and $h_j(V)\subset V$. As for the map $h_j$, in addition to {\it local}
functions  one could consider some {\it nonlocal} operators as well,
possibly defined via convolutions.

We introduce a variational formulation of problem 
\eqref{eq:00} in $(V,H,V^*)$ by letting the operators $A:H \to 2^H$
and $F:(0,T)\times V \to H$ and the functional $\haz B: V \to
(-\infty,+\infty]$ be defined as
\begin{align*}
 &\xi \in A(w) \ \Leftrightarrow \ \xi \in \alpha(w) \ \text{a.e. in
  $\OO$}\,, &&\xi,w\in H\,,\\[2mm]
&F(t,u) = f(t,u, \nabla u) \ \text{a.e. in
  $\OO$}\,, &&u\in V\,,\\
& \haz B(u) := 
     \displaystyle \int_{\OO} \left(\haz \beta_1(\nabla u(x)) + \haz \beta_0(u(x)) \right)\,
  \d x\,, &&u\in V\,.
\end{align*}
The assumptions on $\alpha$ entail that $A$ is strongly maximal
monotone and linearly bounded with $0\in A(0)$, so that Assumption {\bf A} holds. As $\haz \beta_1$ and $\haz \beta_0$
are smooth and $\haz \beta_1$ has
$p$ growth 
Assumptions {\bf B1--B2}, follow as well. The uniform Lipschitz
continuity 
of $f$ entails Assumption {\bf F} with $h_F(\cdot)=f(\cdot,0,0)\in
L^1(0,T)$. Under the above provisions, the abstract relation
\eqref{eq'}, complemented by the initial condition, is the variational
formulation of \eqref{eq:00} in $(V,H,V^*)$. 
As the assumptions of Subsection \ref{ssec:assumptions} are
satisfied, Theorem~\ref{thm1} entails the existence of solutions. In
particular, we have the following.

\begin{cor}
Under the assumptions of this section, problem \eqref{eq:00}--\eqref{eq:002}
admits an analytically strong martingale solution in the sense of Theorem~\emph{\ref{thm1}}.
\end{cor}


\section*{Acknowledgement}
LS is funded by the Austrian Science Fund (FWF) through the
Lise-Meitner project M\,2876. US is partially supported by 
the Austrian Science Fund (FWF) through projects F\,65, I\,4354,
P\,32788, and by the Vienna Science and Technology Fund through
project MA14-009.


\bibliographystyle{abbrv}


\end{document}